\documentclass[12pt]{amsart}
\usepackage[a4paper, margin=1in]{geometry}

\usepackage{amsmath,amsthm}
\usepackage{amsfonts,amstext,amssymb, mathtools, comment}
\RequirePackage{natbib}
\RequirePackage[colorlinks,citecolor=blue,urlcolor=blue]{hyperref}

\newcommand{\levy}{L\'{e}vy }
\newcommand{\p}{{\mathbb P}}

\newcommand{\e}{{\mathbb E}}
\newcommand{\D}{{\mathrm d}}
\newcommand{\R}{{\mathbb R}}

\renewcommand{\a}{\alpha}
\renewcommand{\b}{\beta}

\newcommand{\ind}[1]{\mbox{\rm\large  1}_{\{#1\}}}

\newcommand{\ep}{\varepsilon}
\newcommand{\epp}{{(\varepsilon)}}
\newcommand{\vaguely}{\stackrel{v}{\rightarrow}}
\newcommand{\X}{{\widehat X}}

\DeclareMathOperator\RV{RV}
\DeclareMathOperator\VAR{var}

\newcommand{\ee}{{\mathrm e}}

\newcommand{\pn}{{\p^{(n)}}}
\newcommand{\Xn}{{X^{(n)}}}
\newcommand{\Oh}{{\mathrm{O}}}

\newcommand{\oPi}{{\overline \Pi}}
\newcommand{\eqd}{{\stackrel{{\rm d}}{=}}}

\newtheorem{theorem}{Theorem}
\newtheorem{proposition}{Proposition}
\newtheorem{corollary}{Corollary}
\newtheorem{lemma}{Lemma}
\newtheorem{remark}{Remark}

\begin{document}

\title[Zooming in on the supremum]{Zooming in on a L\'evy process at its supremum}
\author[J. Ivanovs]{Jevgenijs Ivanovs}
\address{Aarhus University}

\begin{abstract}
Let $M$ and $\tau$ be the supremum and its time of a L\'evy process $X$ on some finite time interval. It is shown that zooming in on $X$ at its supremum, that is, considering $((X_{\tau+t\ep}-M)/a_\ep)_{t\in\mathbb R}$ as $\ep\downarrow 0$, results in $(\xi_t)_{t\in\mathbb R}$ constructed from two independent processes having the laws of some self-similar L\'evy process $\X$ conditioned to stay positive and negative. This holds when $X$ is in the domain of attraction of $\X$ under the zooming-in procedure as opposed to the classical zooming out of~\cite{lamperti}. As an application of this result we establish a limit theorem for the discretization errors in simulation of supremum and its time, which extends the result of~\cite*{asmussen_glynn_pitman1995} for the Brownian motion. 
Additionally, complete characterization of the domains of attraction when zooming in on a \levy process at~0 is provided.
\end{abstract}

\maketitle

\section{Introduction}\label{sec:intro}
The law of the supremum of a \levy process $X$ over a fixed time interval $[0,T]$ plays a key role in various areas of applied probability such as risk theory, queueing, finance and environmental since, to name a few. In particular, it is closely related to first passage (ruin) times, as well as to the distribution of the reflected (queue workload) process. 
Furthermore, this law is essential in pricing path-dependent options such as lookback and barrier options~\citep{broadie_glasserman_kou_old}. 
There are only few examples, however, where the law of the supremum is available in explicit form. More examples are known when $T$ is an independent exponential random variable, see, e.g.,~\citep{lewis_mordecki} and~\citep{kuznetsov_WH}, but this essentially corresponds to taking Laplace transform over time horizon~$T$.
For various representations and estimates of the law of the supremum see the works of~\cite{chaumont_supremum,suprema_bounds,palmowski_supremum,subordinatedBM}
and references therein.

An obvious way to evaluate the law of the supremum is to perform Monte Carlo simulation using a random walk approximation of the \levy process. In other words, the \levy process is simulated on a grid with a small fixed time increment $\ep>0$ which, of course, assumes that $X_{\ep}$ can be simulated efficiently. Even though alternative simulation methods exist~\citep{kyprianou_simulation}, we focus on this obvious discretization scheme and aim at characterizing the limiting behaviour of the discretization or monitoring error. Further motivation comes from the fact that discrete-time models may be more natural in practice, whereas related continuous-time models may admit an explicit solution, see~ \citep{broadie_glasserman_kou} considering such approximations of discrete-time option payoffs. Finally, this setup is consistent with the influential field of high frequency statistics where it is assumed that an It\^o semimartingale is observed at equidistant times with time lag tending to zero~\citep{jacod_protter}.

Define the supremum of $X$ and its discretized counterpart
\[M:=\sup\{X_t:t\in[0,T]\},\qquad M_\ep:=\max\{X_{i\ep}:i=0,\ldots,\lfloor T/\ep\rfloor\}\] and let $\Delta_\ep=M-M_\ep\geq 0$ be the discretization error. The (last) times of the supremum and the maximum are denoted by $\tau$ and $\tau_\ep$, respectively.
In the case when $X$ is a Brownian motion with variance $\sigma^2$ and drift~$\gamma$,~\cite{asmussen_glynn_pitman1995} showed the following weak convergence:
\begin{equation}\label{eq:AGP}\Delta_\ep/(\sigma\sqrt\ep)\Rightarrow V,\qquad \text{ as }\ep\downarrow 0,\end{equation}
where $V$ is defined using two independent copies of a 3-dimensional Bessel process and an independent uniform time shift.
It is intuitive that~\eqref{eq:AGP} continues to hold if $X$ is replaced by an independent sum of a Brownian motion and a compound Poisson process, which is indeed true as shown by~\cite{dia_lamberton2011}.
Despite numerous follow-up works and importance of~\eqref{eq:AGP} in various applications, the limiting behaviour of $\Delta_\ep$ is not known for a general L\'evy process~$X$. In fact, most of the related works are concerned with asymptotic expansions of the expected error~$\e \Delta_\ep$, see~\citep{hanssen_leeuwaarden}, \citep{dia_thesis}, \citep{chen_thesis} and  \citep{dia_lamberton2011}.

In this paper we establish a functional limit theorem for 
$(X_{\tau+t\ep}-M)/a_\ep$, where $a_\ep>0$ and $\ep\downarrow 0$, on the Skorokhod space of two-sided paths, which corresponds to zooming in on the \levy process $X$ at its supremum, see Theorem~\ref{thm:zooming}. The limit process $\xi$ for positive times has the law of a certain self-similar \levy process $\X$ conditioned to be negative, whereas for negative times it is the negative of $\X$ conditioned to be positive. It is required for this limit theorem that $X$ is in the domain of attraction of $\X$ (with a scaling function $a_\ep$) under the zooming-in procedure as opposed to the classical zooming-out of~\cite{lamperti}. It is noted that zooming-in and zooming-out domains are very different, and the former is determined by the behaviour of $X$ at~0, see Theorem~\ref{thm:domains}. Finally, a general version of~\eqref{eq:AGP} is provided in Theorem~\ref{thm:main} which additionally includes the scaled difference  of suprema times $(\tau-\tau_\ep)/\ep$. In particular, it is shown that~\eqref{eq:AGP} holds whenever the Brownian component is present, i.e.~$\sigma>0$ in the L\'evy-Khintchine formula~\eqref{eq:psi}.

Let us briefly discuss some additional related literature. In the study of extremes of Gaussian processes~\citep{piterbarg} it is standard to assume that the process of interest locally behaves as a fractional Brownian motion or, more generally, as a self-similar centered Gaussian process. In the context of \levy processes, \cite{barczy_bertoin} obtained a somewhat related functional limit theorem by starting the process (with a negative drift) at $x\rightarrow -\infty$, conditioning on having a positive supremum, and shifting at the instant of the supremum. 
Finally, it is noted that our problem does not fit into the standard framework of high frequency statistics~\citep{jacod_protter}, because the discretization error $\Delta_\ep$ can not be easily retrieved from the difference of $X$ and its discretized version. 

This paper is organized as follows. \S\ref{sec:prelim} is devoted to preliminaries on \levy processes, self-similar processes, processes conditioned to stay negative, as well as post-supremum processes. In \S\ref{sec:lamperti} we present the result of~\cite{lamperti} but for zooming in instead of zooming out, and then specialize to the case of \levy processes. Complete characterization of the respective domains of attraction together with some noteworthy examples is given in~\S\ref{SEC:DOMAINS}.
A general invariance principle for \levy processes conditioned to stay negative is stated in~\S\ref{sec:invariance}, and the main results of this paper are given in~\S\ref{sec:zooming}. Appendices contain proofs of the results from \S\ref{SEC:DOMAINS} and~\S\ref{sec:invariance}, which are partly known in the literature.


\section{Preliminaries}\label{sec:prelim}
\subsection{Regular variation}
We write $f\in\RV_\a,\a\in\mathbb R$ and say that $f$ is regularly varying at~0 with index~$\a$ if $f$ is a positive measurable function on $(0,\delta)$ for some $\delta>0$ such that $f(x\ep)/f(\ep)\to x^{\a}$ as $\ep\downarrow 0$ for all $x>0$, see~\citep{bingham_regular}. If $f\in\RV_\a$ then $F(t)=f(1/t)$ is regularly varying at $\infty$ with index $-\a$, which allows to convert results from one setting to another. Throughout this paper we consider regular variation at 0 unless specified otherwise.

\subsection{Canonical notation}\label{sec:canonical}
Let $\Omega$ be the set of two-sided paths $\omega:\mathbb R\mapsto\mathbb R\cup\{\dagger\}$ such that 
\[\omega_t=\begin{cases}\omega'_t,&\text{ for }t\in[a,b),\\
\dagger,&\text{ otherwise},\end{cases}\]
for some $a\leq b$ and a two-sided c\'adl\'ag path $\omega':\mathbb R\mapsto\mathbb R$. It will be assumed that $\R\cup\{\dagger\}$ is one-point compactification of the real line, i.e., $\dagger$ is the point at infinity. Furthermore, it is convenient to assume that any algebraic operation involving~$\dagger$ results in~$\dagger$, i.e., $\dagger-x=\dagger$. For a usual path defined on $[0,\infty)$ we put $\omega_t=0$ for all~$t<0$ which will be convenient in the following. Additionally, we may want to terminate the path $\omega$ at some non-negative time $T$, and then we put $\omega_t=\dagger$ for all $t\geq T$. .

We equip $\Omega$ with the extended Skorokhod $J_1$ topology~\citep{whitt}, so that a sequence of two-sided paths converges to some $\omega\in\Omega$ if the restrictions to $[a,b]$ converge for all $a<b$ such that $a,b$ are the continuity points of~$\omega$. We let $X$ be the canonical process: $X_t(\omega)=\omega_t$, and let $\p$ be a probability measure on $\Omega$ with its Borel $\sigma$-algebra $\mathcal F$ under which $(X_t)_{t\geq 0}$ is a \levy process adapted to a usual filtration $(\mathcal F_t)_{t\geq 0}$.
Additionally, we write $\p_x$ for the law of this process issued from~$x$. We say that $X$ is b.v.\ (ub.v.) if $\p$-almost all paths of $X$ are of bounded (unbounded) variation on compacts. 

\subsection{\levy processes}\label{sec:levy}
Consider a \levy process $(X_t)_{t\geq 0}$ and let $\psi(\theta)$ be its \levy exponent: $\e \ee^{\theta X_t}=\ee^{\psi(\theta)t},t\geq 0$ for at least purely imaginary~$\theta$. Standard textbooks on this topic are~\citep{bertoin,sato,kyprianou}.
 The L\'evy-Khintchine formula states that
\begin{equation}\label{eq:psi}\psi(\theta)=\gamma\theta+\frac{1}{2}\sigma^2\theta^2+\int_{\mathbb R}\left(\ee^{\theta x}-1-\theta x\ind{|x|<1}\right)\Pi(\D x),\end{equation}
where $\gamma\in\mathbb R,\sigma\geq 0$ and $\Pi(\D x)$ is a Radon measure on $[-\infty,0)\cup(0,\infty]$ satisfying $\int_{\mathbb R} (x^2\wedge 1)\Pi(\D x)<\infty$. When $\int_{-1}^1 |x|\Pi(\D x)<\infty$ this formula can be rewritten as
\begin{equation}\label{eq:psi_1}\psi(\theta)=\gamma'\theta+\frac{1}{2}\sigma^2\theta^2+\int_{\mathbb R}\left(\ee^{\theta x}-1\right)\Pi(\D x),\end{equation} which corresponds to an independent sum of a drifted Brownian motion with mean $\gamma'$ and variance $\sigma^2$, and a pure jump b.v.\ process. 

Throughout this work we exclude the trivial process which is equal to 0 identically. 
Concerning the behaviour of $X$ for large $t$, we recall that only the following three possibilities can occur as $t\rightarrow\infty$: (i) $X_t\rightarrow\infty$, (ii) $\liminf_t X_t=-\infty$ and $\limsup_t X_t=\infty$, (iii) $X_t\rightarrow-\infty$ a.s., where in case (ii) we say that $X$ \emph{oscillates}. 


Often it is convenient to consider a \levy process $X$ \emph{killed} (sent to $\dagger$) at an independent exponential time $e_q$ of rate~$q>0$. This is the only way of killing which preserves stationarity and independence of increments, and so it leads to a natural generalization of a \levy process. We often keep $q\geq 0$ implicit, but write $\p^q,\psi^q$ when it is necessary to stress that the corresponding \levy process is killed at rate~$q$. The L\'evy-Khintchine formula~\eqref{eq:psi} is extended to killed \levy processes by putting $\psi^q(\theta)=\psi(\theta)-q$ so that $\e^q(e^{\theta X_t};X_t\neq \dagger)=e^{\psi^q(\theta)t}.$

Finally, we define the overall supremum and its (last) time:
\[\overline X:=\sup_{t\geq 0}\{X_t:X_t\neq \dagger\},\qquad \overline G:=\sup\{t\geq 0:X_t=\overline X\text{ or }X_{t-}=\overline X\},\] 
so that $\overline G=\infty$ when $\overline X=\infty$. The latter occurs when $X$ drifts to $\infty$ or oscillates, in which case $X$ must be non-killed. Additionally, we let $\underline X:=\inf_{t\geq 0}\{X_t:X_t\neq \dagger\}$ to denote the overall infimum.

\subsection{Self-similar processes}\label{sec:self_similar}
A process $(X_t)_{t\geq 0}$ is called \emph{self-similar with index $H>0$} if for all $u>0$ it holds that 
\begin{equation}\label{eq:self_similar}(X_{ut})_{t\geq 0}\eqd(u^H X_{t})_{t\geq 0},\end{equation} 
and in particular $X_0=0$ a.s. The index $H$ is unique when $X$ is not identically $0$ or~$\dagger$; both are said to be trivial in the following. Standard textbook references are~\cite[Ch.\ 7]{samorodnitsky_taqqu} and~\citep{embrechts_selfsimilar}.

Suppose that $X$ is a non-trivial self-similar \levy process then necessarily $\a:=1/H\in(0,2]$ and $q=0$ (no killing). 
The following is an exhaustive list of self-similar \levy processes:
\begin{itemize}
\item[(i)] Brownian motion: $\gamma=0,\sigma>0,\Pi=0$, in which case~$\a=2$;
\item[(ii)] Linear drift process: $\gamma\neq 0,\sigma=0,\Pi=0$, in which case $\a=1$;
\item[(iii)] Strictly $\a$-stable \levy process for $\a\in(0,2)$: $\sigma=0,$
\begin{equation}\label{eq:psi_stable}\Pi(\D x)=\ind{x>0}c_+x^{-1-\a}\D x+\ind{x<0}c_-|x|^{-1-\a}\D x\end{equation}
for some $c_\pm\geq 0,c_++c_->0$, and, additionally, 
\begin{align*}
\gamma=(c_+-c_-)/(1-\a)\qquad\text{if }\a\neq 1,\\
c_+=c_-,\qquad\text{if }\a=1,\nonumber
\end{align*}
see~\cite[Thm.\ 14.7 (iv)--(vi)]{sato}.
\end{itemize}
The linear drift process in (ii) is often excluded from consideration. This simple process, however, is needed for completeness of the limit theory presented in Theorem~\ref{thm:lamperti}, see also Remark~\ref{rem:trivial_degenerate}. It is not always possible to subtract a linear drift to get another (stable) limit process, see~\S\ref{sec:ex_stable}. Furthermore, in our application to the study of supremum such a transformation would completely change the problem.

Suppose $X$ is a self-similar \levy process which is not a linear drift process. Then $X$ is b.v.\ if and only if $\a\in(0,1)$, in which case we may use the representation~\eqref{eq:psi_1} with~$\gamma'=0$ and $\sigma=0$. In particular, if $X$ is monotone then necessarily $\a<1$, and so it is a pure jump process with all the jumps of the same sign.
Finally, if $X$ is not monotone then the point 0 is regular for $(-\infty,0)$ and $(0,\infty)$, see~\cite[Thm.\ 6.5]{kyprianou}. In this case, by self-similarity, the process $X$ must be oscillating and so $\overline X=\infty$ and $\underline X=-\infty$.

\subsection{Processes conditioned to stay negative}
For any $x<0$ we may define the law of a \levy process $X$ started in $x$ and conditioned to stay negative:
\[\p_x^\downarrow(\cdot):=\p_x(\cdot|\overline X<0)\]
unless $\p(\overline X=\infty)=1$, because then we would condition on the event of zero probability. In general, we first consider a killed process and then take the limit:
\begin{equation}\label{eq:p_uparrow}\p_x^\downarrow(B):=\lim_{q\downarrow 0}\p_x^q(B|\overline X<0)\end{equation}
for all $B\in\mathcal F_T,T\in[0,\infty)$,
which defines a probability law~\citep{chaumont}. 
It is well known that the process under $\p_x^\downarrow$ is a Markov process on $(-\infty,0)$ with a Feller semigroup, say $p^\downarrow_t(x,\D y)$.
This process has infinite life time if and only if the original \levy process $X$ satisfies $\underline X=-\infty$, i.e.\ $X$ either drifts to $-\infty$ or oscillates. Finally, it is standard to express the semigroup $p^\downarrow_t(x,\D y)$ as Doob's h-transform of $X$ killed at the entrance time into $[0,\infty)$, see~\eqref{eq:doob_h} for the precise expression.

It is crucial to take the limit in~\eqref{eq:p_uparrow} along independent exponential times, that is, the limit of conditioned killed \levy processes, because deterministic times may result in a different limit law. In particular, when $X\rightarrow \infty$ the life time of the process under $\p_x^\downarrow$ is finite, whereas deterministic times necessarily lead to an infinite lifetime if the corresponding limit law exists, see also~\citep{hirano_conditioning}.
When $X$ oscillates, we may alternatively condition on $X$ exiting $(-y,0)$ through $-y$ and then letting $y\rightarrow\infty$, see~\cite[Rem.\ 1]{chaumont}. Finally, according to~\cite[Rem.\ 1]{chaumont_96}, for a non-monotone self-similar \levy process we may also take the limit along deterministic times:
\[\p_x^\downarrow(B)=\lim_{t\rightarrow \infty}\p_x(B|X_s<0\,\forall s\in[0,t]).\]



\subsection{Post-supremum processes}\label{sec:post_sup}
Unless $\overline X=\infty$ we consider the post-supremum process $(X_{\overline G+t}-\overline X)_{t\geq 0}$, and denote its law by~$\p^\downarrow$ (there is no subscript as compared to the conditional law $\p^\downarrow_x$). In general, we consider $X$ on a finite time interval $[0,T]$ and the corresponding post-supremum process. Then we take $T\rightarrow\infty$ to define the law~$\p^\downarrow$, see~\citep{bertoin_splitting}, where it is also shown that 
the process under $\p^\downarrow$ is Markov with transition semigroup $p^\downarrow_t(x,\D y)$ for any $x,y<0$ and $t\geq 0$. This explains the notation for the law of the post-supremum process; moreover $\p^\downarrow$ is also called the law of $X$ conditioned to stay negative. 
If $X$ is such that 0 is regular for $(-\infty,0)$ then the process under $\p^\downarrow$ starts at~0 and leaves it immediately, but otherwise it starts at a negative value having a certain distribution, see~\citep{chaumont}. In the latter case the post-supremum process may also be identically $\dagger$ with positive probability. 

It should be noted that some of the cited results are stated for non-killed processes, but their extension to killed \levy processes is straightforward.
Furthermore, in the analogous way we define the laws $\p^\uparrow_x,x>0$ and $\p^\uparrow$ corresponding to the \levy process conditioned to stay positive and the post-infimum process, respectively; one may easily obtain these laws by considering $-X$. 

In this paper we will focus on a self-similar \levy process~$\X$ with law $\p$ arising as a weak limit when zooming in on~$X$. Recall that such $\X$ oscillates when non-monotone and hence both $\widehat\p^\uparrow$ and $\widehat\p^\downarrow$ are defined as the limit laws of finite time post-infimum and post-supremum processes, respectively. Furthermore, even for non-oscillating non-killed process one of the above laws is defined as a limit.


\section{The result of Lamperti for zooming in}\label{sec:lamperti}
\subsection{Zooming out -- the classical theory}\label{sec:zooming_out}
Consider an arbitrary stochastic process $X$, and assume that $(X_{\eta t}/a_\eta)_{t\geq 0}$ has a stochastically continuous, non-trivial limit $\X$ as $\eta\rightarrow\infty$ for some scaling function $a_\eta>0$, in the sense of convergence of finite dimensional distributions. 
\cite{lamperti} showed that necessarily $\X$ is a self-similar processes, see \S\ref{sec:self_similar}.
In fact, \cite{lamperti} considered a more general scaling of the form $X_{\eta t}/a_\eta+b_\eta$ while assuming that $\X_t$ is non-degenerate for every~$t$. In that case $b_\eta\to b$ and so one may as well drop $b_\eta$ which would still result in a stochastically continuous limit process. 

The above rescaling may be seen as zooming out on the process~$X$, and a classical example is the generalized Donsker's theorem, where $X_t=\sum_{i=1}^{\lfloor t\rfloor}\zeta_i$ for an i.i.d.\ sequence of random variables~$\zeta_i$, see e.g.~\cite[Ch.\ 4]{whitt_book}. In this case all the possible non-trivial limits of $(X_{\eta t}/a_\eta)_{t\geq 0}$ are given by the class of self-similar \levy processes $\X$ with the necessary and sufficient condition~\cite[Thm.\ 16.14]{kallenberg}  being
\begin{equation}\label{eq:sums}\sum_{i=1}^n\zeta_i/a_n\Rightarrow \X_1.
\end{equation} Strict domains of attraction, when the index of stability is different from~1, can be obtained from non-strict  domains characterized in~\cite[Thm.\ 7.35.2]{gnedenko_kolmogorov}, but see also~\cite[Thm.\ 8.3.1]{bingham_regular} and comments following it. The case of strictly $1$-stable law is substantially different and its complete analysis can be found in a rather unknown work of~\cite{shimura}. Finally, characterization of the strict domain of attraction to a non-zero constant is required for the complete picture, see Remark~\ref{rem:trivial_degenerate}. Such result is stated in Appendix~\ref{app:classical}, but see also~\cite[Thm.\ VII.7.3]{feller} for the case of positive random variables.

\subsection{Zooming in}
In this paper, however, we are interested in the opposite scaling of time and space, that is, in zooming in on the process~$X$:
\begin{equation}\label{eq:zoom_in}
\left(X_{\ep t}/a_\ep\right)_{t\geq 0}\stackrel{{\rm fd}}{\Rightarrow} (\X_t)_{t\geq 0}\qquad \text{ as }\ep\downarrow 0,
\end{equation}
and the convergence is in the sense of finite dimensional distributions. 
 Surprisingly, to the best of author's knowledge, this regime has not been properly addressed in the literature. 
By a slight adaptation of the arguments in~\citep[Thm.\ 2]{lamperti}, but see also~\cite[Thm.\ 8.5.2]{bingham_regular}, we get the following result.
\begin{theorem}\label{thm:lamperti}
Assume that~\eqref{eq:zoom_in} holds for a stochastically right-continuous, non-trivial process~$\X$.
Then $\X$ is self-similar with some index $H>0$ as defined in~\eqref{eq:self_similar} and $a_\ep\in \RV_H$ as $\ep\downarrow 0$.
\end{theorem}

Note that $a_\ep\to 0$ and so it must be that $X_0=0$ a.s.
Similarly to the classical case, the more general scaling of the form $(X_{\ep t}+b_\ep)/a_\ep$ is superfluous. It allows for processes $X$ started at some deterministic $x$, but the same can be achieved by simply considering $(X_{\ep t}-x)/a_\ep$. Finally, it should be stressed that Theorem~\ref{thm:lamperti} can be extended by considering the time interval $(0,\infty)$ instead of $[0,\infty)$ in~\eqref{eq:zoom_in}, in which case there is an additional possibility that $a_\ep\in\RV_0$ and $(\X_{ut})_{t>0}\eqd(\X_t+b\log u)_{t>0}$ for some $b\in\R$ and all~$u>0$.

\begin{remark}\label{rem:trivial_degenerate}
In the setting of an arbitrary positive affine scaling one assumes that the limit process is non-degenerate for some~$t>0$, i.e., the distribution of $\X_t$ does not concentrate at a point, see~\cite[Ch.\ 8.5]{bingham_regular}. For the above scaling, however, it is sufficient that the limit process is non-trivial. The reason is that in the corresponding Convergence to Types Lemma~\ref{lem:types} it is only required that one random variable does not concentrate at~0. In particular, the linear drift process is not excluded in the statement of Theorem~\ref{thm:lamperti}.
\end{remark}
\begin{lemma}[Convergence to Types]\label{lem:types}
Suppose that for some $a_n,a'_n>0$ and random variables $X_n,X,X'$,
\[X_n/a_n\Rightarrow X\quad\text{ and }\quad X_n/a'_n\Rightarrow X',\qquad n\to \infty,\] and $\p(X=0)<1$.
Then $a_n/a'_n\to u\in[0,\infty)$ and $X'\eqd uX$.
\end{lemma}
\begin{proof}
Adapt the proofs of~\cite[Thm. 2.10.1 and Thm. 2.10.2]{gnedenko_kolmogorov}. 
\end{proof}
Furthermore, Convergence to Types result implies that if Theorem~\ref{thm:lamperti} holds with another scaling function $a_\ep'>0$ and non-trivial limit process $\X'$ then necessarily 
\begin{equation}\label{eq:types}
a_\ep/a'_\ep\to u\in(0,\infty)\quad\text{ and }\quad (\X'_t)_{t\geq 0}\eqd(u\X_t)_{t\geq 0}.
\end{equation}

\subsection{Zooming in on a \levy process}
Let us specialize~\eqref{eq:zoom_in} to the case when $X$ is a \levy process with the \levy exponent $\psi$. It is clear that stationarity and independence of increments must be preserved by the limit process, and so $\X$ must be a \levy process; its \levy exponent is denoted by~$\widehat\psi$.
Now the convergence in~\eqref{eq:zoom_in} extends to the weak convergence on the Skorokhod space~\cite[Cor.\ VII.3.6]{jacod_shiryaev}, and it is equivalent to 
\begin{equation}\label{eq:psi_convergence}\psi^\epp(\theta)=\ep\psi(\theta/a_\ep)\rightarrow\widehat\psi(\theta),\qquad \text{as }\ep\downarrow 0
\end{equation}
for all purely imaginary~$\theta$, where $\psi^\epp$ is the \levy exponent of the \levy process $X^\epp_t=X_{\ep t}/a_\ep$.
According to Theorem~\ref{thm:lamperti}, if $\X$ is non-trivial then it is $1/\a$-self-similar \levy process, see \S\ref{sec:self_similar}, and $a_\ep\in\RV_{1/\a}$ for some $\a\in(0,2]$. 
Necessary and sufficient conditions for the convergence in~\eqref{eq:psi_convergence} are provided in \S\ref{SEC:DOMAINS}. In this regard it is noted that there exist \levy processes such that no scaling function $a_\ep>0$ satisfies~\eqref{eq:psi_convergence}, i.e., such \levy processes do not have a non-trivial limit under zooming in. A simple example is given by a compound Poisson process. It should be stressed that throughout this paper the limits in~\eqref{eq:zoom_in} and~\eqref{eq:psi_convergence} are assumed to hold for all sequences $\ep_n\downarrow 0$. Alternatively, one may talk about partial attraction by requiring the above for some sequence $\ep_n$ only, see~\cite[\S 37]{gnedenko_kolmogorov} and~\citep{maller_strassen}.

We conclude by a simple but important observation.
\begin{lemma}\label{lem:regularity}
Assume that~\eqref{eq:psi_convergence} holds for some non-trivial $\X$.
If $X$ is such that 0 is irregular for $(-\infty,0)$ or for $(0,\infty)$ then $\X$ must be increasing or decreasing, respectively.
\end{lemma}
\begin{proof}
Assume that $0$ is irregular for $(-\infty,0)$. Then with arbitrarily high probability $X_t\geq 0$ for all $t\in[0,h]$, where $h>0$ is small enough, but then $X^\epp_t\geq 0$ for all $t\in[0,h/\ep]$. Using Skorokhod's representation theorem we conclude that $\X$ must be non-negative. This completes the proof of the first statement and the second one follows by considering~$-X$. 
\end{proof}
Importantly, the case when 0 is regular for both $(-\infty,0)$ and $(0,\infty)$ does not in general imply that $\X$ is non-monotone, see \S\ref{sec:ex_stable} for an example.

\section{Domains of attraction when zooming in on a \levy process}\label{SEC:DOMAINS}
In this section for every self-similar \levy processes $\X$, see \S\ref{sec:self_similar}, we provide necessary and sufficient conditions on the characteristics of $X$ so that the limit in~\eqref{eq:psi_convergence} holds true, and also supply the associated scaling function~$a_\ep$. Recall from~\eqref{eq:types} that for any process $X$ the limit $\X$ and the scaling function $a_\ep$ are (asymptotically) unique up to a deterministic factor.
As before, the \levy triplet of $X$ is denoted by $(\gamma,\sigma,\Pi)$, see \S\ref{sec:levy}. Moreover, for a b.v.\ process we use the linear drift~$\gamma'$. The quantities corresponding to~$\X$ are denoted by $\widehat \gamma,\widehat \sigma,\widehat c_\pm$ and so on.

The following zooming-in theory is rather similar to the classical zooming-out theory and the characterization of the strict domains of attraction for sums of i.i.d.\ random variables, see~\cite[Thm.\ 7.35.2]{gnedenko_kolmogorov} or~\cite[Thm.\ 8.3.1]{bingham_regular}, as well as~\citep{shimura}. Instead of conditions on the tails of the distribution of a random variable, in zooming-in context one needs to consider the small-time behaviour of~$X$. 
Characterization of the domains of attraction to a Brownian motion and a linear drift process are due to~\cite{doney_maller_stability}, but see the comments following Theorem~\ref{thm:domains}. Conditions for attraction to strictly stable \levy processes are not readily available in the literature, even though non-strict domains have been characterized by~\cite{maller_mason}. Somewhat related scaling limits of normalized small jump processes are studied by~\cite{asmussen_rosinski} and~\cite{covo}. Additionally, it is noted that the literature on various aspects of small-time behaviour of \levy processes is extensive, see the works of~\cite{doney_fluctuations,bertoin_doney_maller,savov,maller_trimmed} and references therein.

The following result presents some simple observations and, in particular, it states that the \levy measure of $X$ can be modified arbitrarily away from~0 without  affecting the limit under zooming in. 
\begin{lemma}\label{lem:easy_domains}
If $\sigma>0$ then~\eqref{eq:psi_convergence} holds with $\widehat \psi(\theta)=\widehat\sigma^2\theta^2/2$ and $a_\ep\sim\sqrt{\ep}\sigma/\widehat\sigma$ for any $\widehat \sigma>0$.

If $X$ is b.v.\ with $\gamma'\neq 0$ then~\eqref{eq:psi_convergence} holds with $\widehat \psi(\theta)=\widehat \gamma\theta$ and $a_\ep\sim \ep \gamma'/\widehat \gamma$ for any $\widehat \gamma\neq 0$ of the same sign as $\gamma'$.

If~\eqref{eq:psi_convergence} holds for~$X$ then it also holds for the independent sum of $X$ and a compound Poisson process, and vice versa.
\end{lemma}
\begin{proof}
It is well known~\citep[Prop.\ I.2]{bertoin} that $\psi(\theta)/\theta^2\rightarrow \sigma^2/2$ as $|\theta|\rightarrow\infty$. Hence for $a_\ep\sim \sqrt \ep\sigma/\widehat \sigma$ we have
\[\psi^\epp(\theta)=\ep\psi(\theta/a_\ep)=\frac{\psi(\theta/a_\ep)}{\theta^2/a_\ep^2}\theta^2\ep/a^2_\ep\rightarrow \widehat\sigma^2\theta^2/2\] as $\ep\downarrow 0$, and the second claim follows similarly.

Concerning the last statement, it is sufficient to show that $\ep\widetilde\psi(\theta/a_\ep)\to 0$ with $\widetilde\psi(\theta)$ corresponding to any compound Poisson process. This is immediate, because such $|\widetilde\psi(\theta)|$ is bounded.
\end{proof}

For a complete characterization of the domains of attraction we define as in~\citep{maller_mason} the truncated mean and truncated variance functions for $x\in(0,1)$:
\begin{align*}
m(x)=\gamma-\int_{x\leq |y|<1}y\Pi(\D y),\qquad v(x)=\sigma^2+\int_{|y|<x} y^2\Pi(\D y),
\end{align*}
as well as the tails of $\Pi$:
\[\oPi_+(x)=\Pi(x,\infty),\qquad \oPi_-(x)=\Pi(-\infty,-x),\qquad \oPi(x)=\oPi_+(x)+\oPi_-(x).\] Note that when $\int_{-1}^1|x|\Pi(\D x)<\infty$ we have an alternative expression for the truncated mean:
\begin{equation}\label{eq:mean2}
m(x)=\gamma'+\int_{|y|< x}y\Pi(\D y).
\end{equation}

\begin{theorem}[Domains of attraction under zooming in]\label{thm:domains}
The following cases hold true with respect to~\eqref{eq:psi_convergence}:
\begin{itemize}
\item[(i)] $X$ is attracted to the Brownian motion with variance $\widehat \sigma$ if and only if 
\[v\in \RV_0\qquad\text{ or equivalently }\qquad x^2\oPi(x)/v(x)\to 0\]
as $x\downarrow 0$, and $a_\ep$ is chosen to satisfy~$a_\ep^2/v(a_\ep)\sim \ep/\widehat\sigma^2$. 
\item[(ii)] $X$ is attracted to the non-zero linear drift $(\widehat \gamma t)_{t\geq 0}$ if and only if 
\[\sigma=0, \quad m(x)/\widehat \gamma\text{ is eventually positive}, \quad x\oPi(x)/m(x)\to 0\]
as $x\downarrow 0$, and $a_\ep$ is chosen to satisfy $a_\ep/m(a_\ep)\sim\ep/\widehat \gamma$. 
\item[(iii)] $X$ is attracted to the strictly $\a$-stable \levy process with parameters $\widehat c_+,\widehat c_-,\widehat \gamma$, see~\eqref{eq:psi_stable}, if and only if 
\begin{itemize}
\item[(a)] $\sigma=0$, and $\gamma'=0$ when $X$ is b.v.,
\item[(b)] $\oPi_\pm\in\RV_{-\a}$ if $\widehat c_\pm>0$, and $\oPi_+(x)/\oPi_-(x)\to \widehat c_+/\widehat c_-$ as $x\downarrow 0$,
\item[(c)] for $\a=1$ it is additionally required that
\begin{equation}\label{eq:a_is_1}
\frac{m(x)}{x\oPi_+(x)}\to \widehat \gamma/\widehat c_+\qquad\text{as }x\downarrow 0,
\end{equation}
\end{itemize}
and $a_\ep$ is chosen to satisfy $\oPi_\pm(a_\ep)\sim \ep^{-1} \widehat c_\pm/\a$ if $\widehat c_\pm>0$.
\end{itemize}

\end{theorem}
\begin{proof}
For completeness we provide proofs of all three cases in Appendix~\ref{app:proofs} using the same machinery, see also the following comments.
\end{proof}

The cases (i) and (ii) are given by~\cite[Thm.\ 2.5 and Thm.\ 2.2]{doney_maller_stability}. In the former result the convergence statements (2.13) and (2.15) are, in fact, equivalent, meaning that seemingly stronger condition (2.16) can be replaced by (2.14). With respect to (iii) it is noted that~\cite[Thm.\ 2.3]{maller_mason} considered $(X_{\ep t}-b_\ep t)/a_\ep\Rightarrow \X_t$ and characterized the respective non-strict domains. Similarly to the classical case, but in the opposite way, no centering is needed for $\a>1$  and in particular for $\a=2$, and  for $\a<1$ we may choose $b_\ep=\gamma'\ep$, whereas the case $\a=1$ is tricky.

To a \levy measure $\Pi$ it is common to associate the index~\citep{blumenthal_getoor} defined by
\[\beta_{BG}:=\inf\{\beta>0:\int_{|x|<1}|x|^\beta\Pi(\D x)<\infty\},\] where necessarily $\beta_{BG}\in [0,2]$.
\begin{corollary}\label{cor:BG}
If $X$ is attracted to $1/\a$-self-similar \levy process in the sense of~\eqref{eq:psi_convergence} then $\beta_{BG}=\a$, unless $\sigma>0$ or $X$ is b.v.\ with $\gamma'\neq 0$. 
\end{corollary}
\begin{proof}
The proof is given in Appendix~\ref{app:proofs}.
\end{proof}
In particular, Corollary~\ref{cor:BG} shows that for $\a>1$ both $X$ and the limit are ub.v.\ processes, and for $\a<1$ both are b.v.\ processes. In the case of $\a=1$ the limit process may be of different type than~$X$, see~\S\ref{sec:ex_stable}.
In the rest of this section we assume that $\sigma=0$ and $\gamma'=0$ if $X$ is b.v.\ process, since otherwise the limit always exists and it is given by the Brownian motion or the linear drift process, see Lemma~\ref{lem:easy_domains}. It is not hard to verify that these two cases are included in (i) and (ii) of Theorem~\ref{thm:domains}, respectively.

\subsection{Comments}
Note that there are two essentially different limit processes corresponding to $\a=1$: linear drift process in~(ii), and 1-stable \levy process in~(iii). In the latter case $m(x)/(x\oPi(x))$ must have a finite limit~\eqref{eq:a_is_1}, whereas in the former case it must go to~$+\infty$ or $-\infty$.

Consider for a moment condition (b) in Theorem~\ref{thm:domains}~(iii). In the case of $\widehat c_\pm>0$ this condition  is equivalent to multivariate regular variation on the cone consisting of two rays, $\mathbb R_+$ and~$\mathbb R_-$, of the function evaluating to $\oPi_+(x)$ and~$\oPi_-(|x|)$, respectively. It is noted that multivariate regular variation is a common property used in characterizing various domains of attraction~\citep{resnick_heavy}. Let us also point out that for any $\a>0$ it is possible to construct an example of positive decreasing $\oPi_\pm\in\RV_{-\a}$ such that also $\oPi\in\RV_{-\a}$ but the balance condition is not satisfied, i.e., $\oPi_+/\oPi_-$ does not have a limit in $[0,\infty]$.

For $X$ attracted to strictly $\a$-stable process it must be that $\oPi\in \RV_{-\a}$. Regular variation of $\oPi$ is not required, however, when $X$ is attracted to (i) Brownian motion or (ii) linear drift process.
Nevertheless, if we assume that $\oPi\in RV_{-\a}$ then necessarily $\a=2$ in~(i) and $\a=1$ in~(ii), see Corollary~\ref{cor:BG} and its proof; it is assumed here that $\sigma=0$ and $\gamma'=0$ for a b.v.\ process.

Finally, let us provide some examples of \levy processes without a non-trivial limiting process under zooming-in. 
Firstly, any b.v.\ process with $\gamma'=0$ and $\oPi\in\RV_0$, including the compound Poisson process, is such.
Secondly, for any $\a\in(0,1)\cup(1,2)$ we may choose a process with $\oPi\in\RV_{-\a}$ which satisfies (a) of Theorem~\ref{thm:domains}~(iii) but does not satisfy the balance condition in~(b). Thirdly, Corollary~\ref{cor:BG} can be employed to provide further examples with a non regularly varying~$\oPi$.

\subsection{Noteworthy examples}
In the boundary cases, when $\a=2$ and especially so when $\a=1$, somewhat surprising examples can be constructed. 
\subsubsection{Process with $\sigma=0$ attracted to Brownian motion}\label{sec:ex_BM}
Take $\Pi(\D x)=x^{-3}\log^{-2}x\D x$ for small $x>0$ and let $\Pi(-\infty,0)=0$ so that
\[v(x)=\int_0^x y^{-1}\log^{-2}y\D y=-1/\log x\in\RV_0.\]
According to Theorem~\ref{thm:domains}~(i) this process is attracted by the Brownian motion.
The scaling function must satisfy $-a_\ep^2\log a_\ep\sim\ep/\widehat \sigma^2$ and in particular $a_\ep/\sqrt \ep\to 0$. 
\subsubsection{Non-strictly 1-stable process is attracted to linear drift}\label{sec:ex_stable}
Let $X$ be a 1-stable process~\eqref{eq:psi_stable} which is not strictly stable, i.e., $c_+\neq c_-$. A simple computation reveals that $\oPi(x)=(c_++c_-)/x$ and $m(x)=\gamma+(c_+-c_-)\log x$. Hence we see that the conditions of Theorem~\ref{thm:domains}~(ii) are satisfied for any $\widehat \gamma$ having the same sign as $(c_--c_+)$. Therefore, a non-strictly 1-stable process is attracted to a linear drift process. The scaling function must satisfy \[-a_\ep/\log a_\ep\sim\ep(c_--c_+)/\widehat \gamma,\] and so $a_\ep/\ep\to \infty$. In this case one may also verify~\eqref{eq:psi_convergence} directly using the above function $a_\ep$ and the analytic representation of $\psi(\theta)$, see~\cite[(14.20) and (14.25)]{sato}.
 This example shows in particular that ub.v.\ process may have a b.v.\ limit, which at first sight may look counter-intuitive: the process $X$ is such that $0$ is regular for both half lines, whereas $0$ is irregular for one half line for the limit process.
Note also that when $c_+=0$ the limit is a positive drift, which intuitively means that under zooming-in we see the drift compensating negative jumps. 
 
\subsubsection{B.v.\ process attracted to strictly 1-stable process}
Let $X$ be b.v.\ process with $\gamma'=0$ and $\oPi_+=\oPi_-\in\RV_{-1}$. A concrete example is obtained by taking $\oPi_+(x)=x^{-1}\log^{-2} x$  for small $x>0$. Now $m(x)=\int_{|y|<x} y\Pi(\D y)=0$ and so~\eqref{eq:a_is_1} holds with $\widehat\gamma=0$. The appropriate scaling function satisfies 
\[a_\ep\log^2 a_\ep\sim \ep/\widehat c_+,\]
and so $a_\ep/\ep\to 0$.
\subsubsection{On the necessity of~\eqref{eq:a_is_1}}\label{example:last}
Let $\sigma=0,\gamma=0$ and $\oPi_\pm(x)=-x^{-1}/\log x$ for small $x>0$ so that $X$ is ub.v.\ process. As in the above example the limit is strictly 1-stable process with $\widehat\gamma=0$. Next, keeping everything else the same let $\oPi_+(x)=-x^{-1}/\log x+\ind{x\leq 1/2}$, which yields $m(x)=-1/2$ and thus $m(x)/\{x\oPi_+(x)\}\to-\infty$. 
In particular, we see that~\eqref{eq:a_is_1} does not hold even though the assumptions (a) and (b) of Theorem~\ref{thm:domains}~(iii) are satisfied. In fact, the limit process must be a negative linear drift, see Theorem~\ref{thm:domains}~(ii).
This may seem to contradict the last item of Lemma~\ref{lem:easy_domains}. Observe, however, that addition of an independent compound Poisson process with \levy measure $\delta_{1/2}$ leads also to modification of $\gamma$ so that $\gamma=1/2$, and in that case the limit is preserved.



\section{Invariance principle for \levy processes conditioned to stay negative}\label{sec:invariance}
Invariance principles for processes derived from random walks is a classical theme in probability~\citep{skorokhod}.
Concerning the case of a random walk conditioned to stay negative the reader is referred to the works of~\cite{caravenna_chaumont}, \cite{invariance_chaumont_doney} and references therein. By the standard approximation argument one may also derive an invariance principle for \levy processes conditioned to stay negative, which is stated below.

Recall from \S\ref{sec:canonical} that we work with two-sided paths taking values in~$\R$ compactified by addition of the absorbing state $\dagger$, and such that $\omega_t=0$ for all $t<0$. This trick allows us to provide a clean formulation of the following functional limit theorem. 
\begin{theorem}\label{thm:convergence}
Let $\Xn$ be a sequence of (possibly killed) \levy processes weakly converging to a \levy process $X$, which is  not a compound Poisson process. Then $\pn_x^\downarrow\Rightarrow\p_x^\downarrow$ for any $x<0$ and  $\pn^\downarrow\Rightarrow\p^\downarrow$, where the latter law may put a positive mass on~$(\dagger)_{t\geq 0}$.
\end{theorem}
If the process $X$ has finite supremum then the above statement follows immediately from the continuous mapping theorem and the fact that~$X$ has a unique time of the supremum. The main difficulty lies in the other case, where the law $\p^\downarrow$ is defined as a limit. In fact, Theorem~\ref{thm:convergence} follows by a standard approximation argument from~\cite[Thm.\ 4]{invariance_chaumont_doney}, at least when $X$ is such that $0$ is regular for both half lines $(-\infty,0)$ and $(0,\infty)$, and the processes $X,\Xn$ are non-killed and do not drift to $-\infty$.
An alternative proof of Theorem~\ref{thm:convergence} is given in Appendix~\ref{app:convergence}.

The assumption of two-sided paths with $\omega_t=0$ for all $t<0$ allows us to avoid the following problem. Suppose that $X$ is such that 0 is irregular for $(-\infty,0)$, but $\Xn$ are such that 0 is regular for $(-\infty,0)$; for example, we may add to $X$ a Brownian motion with diminishing variance. 
Then $X$ leads to the post-supremum process starting  at a negative level, whereas for $\Xn$ such process starts at~0 and then quickly jumps to a negative level when $n$ is large. The assumption that these processes are fixed at 0 for negative times ensures the claimed convergence in the Skorokhod topology. A similar problem but with a different solution appears in~\citep[Thm.\ 2]{chaumont}. 
Finally, the assumption of Theorem~\ref{thm:convergence} that $X$ is not a compound Poisson process is essential, and a counter example can be easily provided by considering $X_t-t/n$ so that the limit of ${\pn}^\downarrow$ is the law of $X$ conditioned to stay non-positive rather than negative. 


\section{Zooming in on the supremum}\label{sec:zooming}
Consider a \levy process $X$ satisfying~\eqref{eq:psi_convergence} for some function $a_\ep\downarrow 0$ and a non-trivial \levy process~$\X$, which then must be self-similar. Necessary and sufficient conditions for such convergence are given in \S\ref{SEC:DOMAINS}.
Letting $\widehat \p$ be the law of $\X$, we consider a non-positive process $\xi$ on $\mathbb R$ specified by: 
\begin{equation}\label{eq:xi}(\xi_t)_{t\geq 0}\text{ has the law }\widehat\p^\downarrow,\qquad (-\xi_{(-t)-})_{t\geq 0}\text{ has the law }\widehat\p^\uparrow,\end{equation} where the two parts are independent, see \S\ref{sec:post_sup}. Note that on the right hand side we reverse both time and space. In other words, when looking at $\xi$ from the point $(0,0)$ backwards in time and down in space we see the law~$\widehat\p^\uparrow$.
According to the discussion in~\S\ref{sec:self_similar} and in~\S\ref{sec:post_sup} we have the following cases:
\begin{itemize}
\item[(a)] $\X$ is non-monotone (thus oscillating) then $\xi$ has doubly infinite life time, and it is continuous at~0 with $\xi_0=0$;
\item[(b)] $\X$ is decreasing then $\xi_t=\dagger\ind{t<0}+\X_t\ind{t\geq 0}$;
\item[(c)] $\X$ is increasing then $\xi_t=-\X_{(-t)-}\ind{t<0}+\dagger\ind{t\geq 0}$.
\end{itemize}
Furthermore, the laws $\widehat\p^\downarrow$ and $\widehat\p^\uparrow$ inherit self-similarity from $\widehat\p$, and so they correspond to self-similar Markov processes, where the former is negative and the latter is positive (when started away from~0). Such processes are well-studied and, in particular, they enjoy the Lamperti representation via the associated \levy processes, see~\cite[Cor.\ 2]{conditioned_stable} specifying the latter.
Note from Theorem~\ref{thm:lamperti} that both parts of $\xi$ indeed must be self-similar (when non-trivial) if $\xi$ is to be a limit process.


\begin{theorem}\label{thm:zooming}
Let $X$ be a \levy process satisfying~\eqref{eq:psi_convergence} for some function $a_\ep\downarrow 0$ and a non-trivial \levy process~$\X$.
Consider $X$ on $[0,T)$ for any $T>0$, and let $M$ and $\tau$ be the supremum and its time, respectively. Then 
\begin{equation}\label{eq:converegence_thm}((X_{\tau+t\ep}-M)/a_\ep)_{t\in\mathbb R}\Rightarrow (\xi_t)_{t\in\mathbb R} \text{ as }\ep\downarrow 0,\end{equation}
where $\xi$ is defined in~\eqref{eq:xi}. 
Furthermore, the convergence in~\eqref{eq:converegence_thm} is mixing~\citep{renyi} in the sense that it is preserved when the left hand side is conditioned on an arbitrary event $B\in\mathcal F$ of positive probability.
\end{theorem}
\begin{proof}
Note that $X$ can not be a compound Poisson process, because then the limit $\X\equiv 0$ is trivial for any function $a_\ep$. Restriction of $X$ to $[0,T)$ is achieved by putting $X_t=\dagger$ for all $t\notin[0,T)$. 
The main idea is to first consider, instead of a deterministic time horizon, an independent exponential time~$T=e_q$ of rate $q>0$. By doing so we obtain a killed \levy process, which satisfies~\eqref{eq:psi_convergence} with the same~$a_\ep$ and $\widehat\psi$, and hence the corresponding killed \levy process $X^\epp$ converges to the same~$\X$.
Observe that 
\[(X_{\tau+t\ep}-M)/a_\ep=X^\downarrow_{t\ep}/a_\ep={X^\epp}^\downarrow_t,\qquad t\geq 0\] is the post-supremum process corresponding to $X^\epp$, and so its law converges to $\widehat\p^\downarrow$ according to Theorem~\ref{thm:convergence}. 
Moreover, it is well known that the pre-supremum process 
\[-(X_{(\tau-t\ep)-}-M)/a_\ep,\qquad t\geq 0\] is independent of the post-supremum process and has the law of the post-infimum process, which follows from time reversal and splitting~\citep{greenwood_pitman}. Another application of Theorem~\ref{thm:convergence}, but for conditioning to stay positive, shows that the limit law is given by $\widehat\p^\uparrow$. Hence we have the joint convergence of post- and pre-supremum processes, which proves~\eqref{eq:converegence_thm} for a random $T=e_q$.
Moreover, when joining the one-sided paths we use the fact that either $\xi_0=0$ or~$\xi_{0-}=0$.

Next, we show that the convergence is mixing (for the exponential time horizon). According to splitting at the supremum and~\cite[Thm.\ 2]{renyi} it is sufficient to establish that
\[\left.\left(X^\downarrow_{t\ep}/a_\ep\right)_{t\geq 0}\right| A \Rightarrow \left(\X^\downarrow_{t}\right)_{t\geq 0}\qquad \text{ as }\ep\downarrow 0,\]
for any  $A\subset \sigma(X^\downarrow_{s_1},\ldots, X^\downarrow_{s_n})$ and any finite collection of times $0<s_1<\cdots<s_n$; and a similar result for the pre-supremum process. Furthermore, according to~\cite{whitt} it is equivalent to show the above weak convergence for restrictions to $t\in[0,r]$ for any $r>0$, since $\X^\downarrow$ is continuous at $r$ a.s. 
In other words, we aim to show that 
\begin{equation}\label{eq:toshow}\e^\downarrow\left(f\{(X_{t\ep}/a_\ep)_{t\in[0,r]}\} g\{X_{s_1},\ldots,X_{s_n}\}\right)\to\widehat \e^\downarrow f\{(X_t)_{t\in[0,r]}\}\e^\downarrow g\{X_{s_1},\ldots,X_{s_n}\}\end{equation}
for bounded continuous functions $f$ and $g$. Letting $\ep>0$ be so small that $r\ep<s_1$ we find by the Markov property of $X^\downarrow$ that the left hand side of \eqref{eq:toshow} is given by
\begin{align*}
&\int_{-\infty}^0\e^\downarrow\left(f\{(X_{t\ep}/a_\ep)_{t\in[0,r]}\};X_{r\ep}/a_\ep\in\D x\right)\e^\downarrow_{xa_\ep}g\{X_{s_1-r\ep},\ldots,X_{s_n-r\ep}\}\\
&=:\int_{-\infty}^0\mu_\ep(\D x)h_\ep(x).
\end{align*}
Similarly, the right hand side of~\eqref{eq:toshow} can be written as
\[\int_{-\infty}^0 \widehat\e^\downarrow\left(f\{(X_{t})_{t\in[0,r]}\};X_{r}\in\D x\right)\e^\downarrow g\{X_{s_1},\cdots,X_{s_n}\}=:\int_{-\infty}^0\mu(\D x)h.\]
As before, Theorem~\ref{thm:convergence} guarantees weak convergence of the finite measures: $\mu_\ep\Rightarrow\mu$. Thus it is left to show that for any $x_\ep\to x$ we have $h_\ep(x_\ep)\to h(x)=h$, which implies~\eqref{eq:toshow} in view of the Skorokhod's representation theorem, but see also~\cite[Thm.\ 3.4.4]{whitt_book}.
Finally, the same argument based on Skorokhod's representation theorem can be used to establish that $h_\ep(x_\ep)\to h$. Firstly, from \cite[Thm.\ 2]{chaumont} we find that $\p^\downarrow_{x_\ep a_\ep}\Rightarrow \p^\downarrow$, because  $a_\ep\to 0$. Secondly, the fact that $X^\downarrow$ does not jump at $s_1,\ldots,s_n$ shows convergence of the corresponding functionals. This concludes the proof for an independent exponential time horizon~$e_q$.

Finally, we extend the result to an arbitrary deterministic~$T>0$.
Consider a bounded continuous functional $f$ on the Skorokhod space of two-sided paths.
Let $F^\epp_T$ and $F^{(0)}$ denote $f$ applied to the paths on the left hand side of~\eqref{eq:converegence_thm} and the right hand side, respectively.
The first parts of the proof show that 
\[q\int_0^\infty \ee^{-qt}\e \left(\left.F^\epp_t\right|B\right)\D t\rightarrow \e^* F^{(0)}=q\int_0^\infty \ee^{-qt}\e^* F^{(0)}\D t,\] that is, the Laplace transforms in~$t$ converge, where $\e^*$ denotes the law of $\xi$ and $e_q$ is taken independent of~$B$. Hence $\e \left(\left.F^\epp_t\right|B\right)\rightarrow \e^* F^{(0)}$ for almost all~$t>0$, implying the corresponding weak convergence. 
If $X$ is such that 0 is regular for $(-\infty,0)$ then $\tau\neq T$ a.s. Thus with arbitrarily high probability we may choose small enough $\delta>0$ such that $T-\tau>2\delta$, and then for any $\ep$ the rescaled post-supremum processes corresponding to $T$ and $T'\in(T-\delta,T)$ coincide at least up to time $\delta/\ep$, which means that the respective Skorokhod distance tends to 0 as $\ep\downarrow 0$; whereas the corresponding pre-supremum processes are identical. It is left to choose $T'$ for which~\eqref{eq:converegence_thm} holds true, and to apply~\cite[Thm.\ 2.7 (iv)]{vanderVaart}.
If 0 is irregular for $(-\infty,0)$ then we use time reversal to translate our supremum problem into infimum problem, and observe that the infimum can  not be achieved at the end point~$T$.
\end{proof}

Let us provide some commentary with respect to Theorem~\ref{thm:zooming}. Assume for a moment that $X$ is such that $0$ is irregular for $(-\infty,0)$. According to Lemma~\ref{lem:regularity} if $X$ is in the domain of attraction of some non-trivial $\X$ then the latter is increasing, and the corresponding limiting post-supremum process is~$(\dagger)_{t\geq 0}$. Indeed, 
the post-supremum process of $X$ starts at a negative value or~$\dagger$, and upon zooming-in it must reduce to identically killed process; recall that $\dagger$ is assumed to be a point at~$\pm\infty$.
A very similar conclusion can be drawn about the case when $0$ is irregular for $(0,\infty)$.


Interestingly, the above behaviour can also be exhibited by a process $X$ for which 0 is regular for both half-lines, and so $X$ is continuous at~$\tau$. For example, consider a 1-stable process with $c_->c_+$, see \S\ref{sec:ex_stable}, in which case we may take $\X_t=t$.
In other words, the corresponding scaling function $a_\ep$ works fine for the pre-supremum process, but is decreasing too fast for the post-supremum process. It may be interesting to find an appropriate scaling function for the latter if such exists. 

Finally, let us show that mixing convergence in Theorem~\ref{thm:zooming} easily leads to further generalizations.
\begin{corollary}\label{cor:random}
The result of Theorem~\ref{thm:zooming} extends to an arbitrary random time interval $[\rho_1,\rho_2)$ and an event $B$, such that on $B$ it holds that $\rho_1<\rho_2<\infty$ and $\tau\notin\{\rho_1,\rho_2\}$. If $\rho_1$ is a stopping time then the latter condition can be weakened to $\tau\neq \rho_2$.
\end{corollary}
\begin{proof}
We may choose $\delta>0$ so small that $\tau\in(\rho_1+\delta,\rho_2-\delta)$ with arbitrarily high probability, where $\tau$ is the time of the supremum of the process restricted to $[\rho_1,\rho_2)$. Using the argument from the last step in the proof of Theorem~\ref{thm:zooming} we find that the claimed result holds on the event $B$ jointly with $\rho_i\in[k_i\delta/2,(k_i+1)\delta/2)$ for some fixed integers $k_1,k_2$ (when it has positive probability) and the above condition on $\tau$. The rest is obvious.
\end{proof}

\subsection{Discretization error}
Next, using Theorem~\ref{thm:zooming} we derive a limit result for the discretization error $\Delta_\ep$ generalizing~\eqref{eq:AGP}, see \S\ref{sec:intro}. Another important ingredient is the old result of~\cite{kosulajeff} stating that the fractional part $\{\tau/\ep\}$ weakly converges to a uniform random variable as $\ep\downarrow 0$ for an arbitrary random variable $\tau$ possessing Lebesgue density.
\begin{theorem}\label{thm:main}
Let $U$ be an independent uniform $(0,1)$ random variable.
Under the conditions of Theorem~\ref{thm:zooming}, for a non-monotone $X$ it holds on the event $\tau\notin\{0,T\}$ that
\[(-\Delta_\ep/a_\ep,(\tau_\ep-\tau)/\ep)\Rightarrow (\max_{i\in\mathbb Z} \xi_{U+i},U+{\rm argmax}_{i\in\mathbb Z} \xi_{U+i})\qquad \ep\downarrow 0.\]
If $\X$ is decreasing or increasing then the limiting pair reduces to $(\X_U,U)$ or $-(\X_U,U)$, respectively. 
\end{theorem}
\begin{proof}
Note that observing $X_t$ at the time instants $i\ep, i\in\mathbb Z$ corresponds to observing $X_{\tau+t\ep}$
at the time instants $\mathbb Z-\{\tau/\ep\}$. 
It is well known~\cite[Thm.\ 6]{chaumont_supremum} that the distribution of $\tau$ has a Lebesgue density on $(0,T)$ and possibly an atom at 0 or at $T$. According to~\cite{kosulajeff}, on the event $\tau\notin\{0,T\}$  we have that $\{\tau/\ep\}\Rightarrow U$.
Furthermore, the mixing convergence in Theorem~\ref{thm:zooming} shows that
\[\left(1-\{\tau/\ep\},((X_{\tau+t\ep}-M)/a_\ep)_{t\in\mathbb R}\right)\Rightarrow \left(U,(\xi_t)_{t\in\mathbb R}\right),\] 
where $U$ and $\xi$ are independent. 
Finally, note that  
$\xi_t$ observed at times $i+U,i\in\mathbb Z$ has a unique maximum.
Furthermore, $\xi$ is continuous at each of the observation instants a.s., and so the continuous mapping theorem completes the proof.
\end{proof}

As a consequence of Theorem~\ref{thm:main} the following can be said about the three cases of Theorem~\ref{thm:domains}.
\begin{itemize}
\item[(i)] If $\sigma>0$ then~\eqref{eq:AGP} holds true: choose $a_\ep=\sigma\sqrt\ep$ and observe that $\X$ is a standard Brownian motion, which implies that the law of $|\xi_t|$ for positive and negative times corresponds to the three-dimensional Bessel process. Moreover, the same limit can be obtained for a process with $\sigma=0$, but then the scaling function $a_\ep$ must satisfy $a_\ep/\sqrt \ep\to 0$, see~\S\ref{sec:ex_BM}.
\item[(ii)]
If $X$ is b.v.\ with $\gamma'\neq 0$ then 
\[\Delta_\ep/(|\gamma'|\ep)\Rightarrow U\qquad\text{ on the event }\tau\notin\{0,T\}.\]
The same limit law can be obtained when $X$ is, e.g., a 1-stable process with $c_+\neq c_-$, see \S\ref{sec:ex_stable}.
\item[(iii)]
A strictly $\a$-stable process $\X$ has two free parameters, one of which can be fixed by an appropriate choice of the scaling function~$a_\ep$. Alternatively, we may use the positivity parameter $\widehat\rho=\p(\X_1>0)$, so that all possible limits are parametrized by the pair~$(\a,\widehat\rho)$ in a certain domain. When $\a\in(0,1)$ we may have $\widehat\rho=0$ or $\widehat\rho=1$ corresponding to a monotone~$\X$.
\end{itemize}

According to Theorem~\ref{thm:lamperti} and Corollary~\ref{cor:BG}, in the case when $\sigma=0$ and $\gamma'=0$ if $X$ is b.v.\ the scaling function must satisfy $a_\ep\in\RV_{1/\beta_{BG}}$, where $\beta_{BG}$ is the corresponding Blumenthal-Getoor index, given that $X$ is in the domain of attraction of some non-trivial $\X$ which then must be $1/\beta_{BG}$-self-similar.

Finally, it is easy to see that the same weak limit as in Theorem~\ref{thm:main} is obtained for $((\underline M-\underline M_\ep)/a_\ep,(\underline\tau-\underline \tau_\ep)/\ep)$ on the event $\underline \tau\notin\{0,T\}$, where $\underline M,\underline \tau$ and $\underline M_\ep,\underline\tau_\ep$ are the infimum of $X$ on $[0,T)$ with its time and their discretized analogues, respectively.

\subsection{Further comments}

As mentioned in \S\ref{sec:intro}, there is quite some interest in the literature in determining the rate of convergence of the expected error $\e \Delta_\ep$ to~0. For example,
\cite{dia_lamberton2011} and~\cite{chen_thesis} showed, respectively, that $\e\Delta_\ep=\Oh(\sqrt \ep)$ if $\sigma>0$, and that $\e\Delta_\ep=\Oh(\ep^r)$ for $r<1/\beta_{BG}$ if $\sigma=0$ and the process is ub.v.
Our results provide a hint on the rate, but do not readily determine it. The reason is that proving uniform integrability of $\Delta_\ep/a_\ep$ seems to be a hard task in general. In some cases the representation of $\Delta_\ep$ based on Spitzer's identity, see~\citep[Eq.\ (3.3)]{asmussen_glynn_pitman1995}, may be useful. Furthermore, it is anticipated that uniform integrability does not hold when the attractor $\X$ is a strictly $\a$-stable \levy process with $\a<1$, which is clearly true when $\X$ is monotone and $\e |\X_U|=\infty$.

Finally, it is possible to apply our results to study the behaviour of $X$ around its first passage and last exit times, instead of the time of supremum. The key result here is the well known path decomposition of the \levy process at these times~\citep{duquesne}. For example, on the event of continuous last exit from some interval $(-\infty,x)$, the post-exit process is independent from the pre-exit process and the former has the law $\p^\uparrow$, whereas the latter when time-reversed has the original law (up to the last exit). Hence using the tools of this paper, and in particular Theorem~\ref{thm:convergence}, we may provide, e.g., a limit result for zooming in on $X$ at its last exit time.

\section*{Acknowledgements}
I am thankful to S{\o}ren Asmussen for his suggestion to reconsider~\citep{asmussen_glynn_pitman1995}. Furthermore, I would like to express my gratitude to Sebastian Engelke, Lanpeng Ji, Jan Pedersen, Mark Podolskij, Orimar Sauri, and two anonymous referees for pointing out various related works.


\begin{thebibliography}{56}
\providecommand{\natexlab}[1]{#1}
\providecommand{\url}[1]{\texttt{#1}}
\expandafter\ifx\csname urlstyle\endcsname\relax
  \providecommand{\doi}[1]{doi: #1}\else
  \providecommand{\doi}{doi: \begingroup \urlstyle{rm}\Url}\fi

\bibitem[Asmussen et~al.(1995)Asmussen, Glynn, and
  Pitman]{asmussen_glynn_pitman1995}
S.~Asmussen, P.~Glynn, and J.~Pitman.
\newblock Discretization error in simulation of one-dimensional reflecting
  {B}rownian motion.
\newblock \emph{Ann. in Appl. Probab.}, pages 875--896, 1995.

\bibitem[Asmussen and Rosi\'nski(2001)]{asmussen_rosinski}
S.~r. Asmussen and J.~Rosi\'nski.
\newblock Approximations of small jumps of {L}\'evy processes with a view
  towards simulation.
\newblock \emph{J. Appl. Probab.}, 38\penalty0 (2):\penalty0 482--493, 2001.

\bibitem[Aurzada et~al.(2013)Aurzada, D\"oring, and Savov]{savov}
F.~Aurzada, L.~D\"oring, and M.~Savov.
\newblock Small time {C}hung-type {LIL} for {L}\'evy processes.
\newblock \emph{Bernoulli}, 19\penalty0 (1):\penalty0 115--136, 2013.

\bibitem[Barczy and Bertoin(2011)]{barczy_bertoin}
M.~Barczy and J.~Bertoin.
\newblock Functional limit theorems for {L}\'evy processes satisfying
  {C}ram\'er's condition.
\newblock \emph{Electron. J. Probab.}, 16:\penalty0 no. 73, 2020--2038, 2011.

\bibitem[Bertoin(1993)]{bertoin_splitting}
J.~Bertoin.
\newblock Splitting at the infimum and excursions in half-lines for random
  walks and {L}\'evy processes.
\newblock \emph{Stochastic Process. Appl.}, 47\penalty0 (1):\penalty0 17--35,
  1993.


\bibitem[Bertoin(1996)]{bertoin}
J.~Bertoin.
\newblock \emph{L\'evy processes}, volume 121 of \emph{Cambridge Tracts in
  Mathematics}.
\newblock Cambridge University Press, Cambridge, 1996.

\bibitem[Bertoin et~al.(2008)Bertoin, Doney, and Maller]{bertoin_doney_maller}
J.~Bertoin, R.~A. Doney, and R.~A. Maller.
\newblock Passage of {L}\'evy processes across power law boundaries at small
  times.
\newblock \emph{Ann. Probab.}, 36\penalty0 (1):\penalty0 160--197, 2008.


\bibitem[Bingham et~al.(1987)Bingham, Goldie, and Teugels]{bingham_regular}
N.~H. Bingham, C.~M. Goldie, and J.~L. Teugels.
\newblock \emph{Regular variation}, volume~27 of \emph{Encyclopedia of
  Mathematics and its Applications}.
\newblock Cambridge University Press, Cambridge, 1987.


\bibitem[Blumenthal and Getoor(1961)]{blumenthal_getoor}
R.~M. Blumenthal and R.~K. Getoor.
\newblock Sample functions of stochastic processes with stationary independent
  increments.
\newblock \emph{J. Math. Mech.}, 10:\penalty0 493--516, 1961.

\bibitem[Broadie et~al.(1997)Broadie, Glasserman, and
  Kou]{broadie_glasserman_kou_old}
M.~Broadie, P.~Glasserman, and S.~Kou.
\newblock A continuity correction for discrete barrier options.
\newblock \emph{Math. Finance}, 7\penalty0 (4):\penalty0 325--349, 1997.


\bibitem[Broadie et~al.(1999)Broadie, Glasserman, and
  Kou]{broadie_glasserman_kou}
M.~Broadie, P.~Glasserman, and S.~G. Kou.
\newblock Connecting discrete and continuous path-dependent options.
\newblock \emph{Finance Stoch.}, 3\penalty0 (1):\penalty0 55--82, 1999.


\bibitem[Caballero and Chaumont(2006)]{conditioned_stable}
M.~E. Caballero and L.~Chaumont.
\newblock Conditioned stable {L}\'evy processes and the {L}amperti
  representation.
\newblock \emph{J. Appl. Probab.}, 43\penalty0 (4):\penalty0 967--983, 2006.


\bibitem[Caravenna and Chaumont(2008)]{caravenna_chaumont}
F.~Caravenna and L.~Chaumont.
\newblock Invariance principles for random walks conditioned to stay positive.
\newblock \emph{Ann. Inst. Henri Poincar\'e Probab. Stat.}, 44\penalty0
  (1):\penalty0 170--190, 2008.


\bibitem[Chaumont(1996)]{chaumont_96}
L.~Chaumont.
\newblock Conditionings and path decompositions for {L}\'evy processes.
\newblock \emph{Stochastic Process. Appl.}, 64\penalty0 (1):\penalty0 39--54,
  1996.


\bibitem[Chaumont(2013)]{chaumont_supremum}
L.~Chaumont.
\newblock On the law of the supremum of {L}\'evy processes.
\newblock \emph{Ann. Probab.}, 41\penalty0 (3A):\penalty0 1191--1217, 2013.


\bibitem[Chaumont and Doney(2005)]{chaumont}
L.~Chaumont and R.~A. Doney.
\newblock On {L}\'evy processes conditioned to stay positive.
\newblock \emph{Electron. J. Probab.}, 10:\penalty0 no. 28, 948--961, 2005.


\bibitem[Chaumont and Doney(2010)]{invariance_chaumont_doney}
L.~Chaumont and R.~A. Doney.
\newblock Invariance principles for local times at the maximum of random walks
  and {L}\'evy processes.
\newblock \emph{Ann. Probab.}, 38\penalty0 (4):\penalty0 1368--1389, 2010.


\bibitem[Chen(2011)]{chen_thesis}
A.~Chen.
\newblock \emph{Sampling error of the supremum of a L{\'e}vy process}.
\newblock PhD thesis, University of Illinois at Urbana-Champaign, 2011.

\bibitem[Covo(2009)]{covo}
S.~Covo.
\newblock On approximations of small jumps of subordinators with particular
  emphasis on a {D}ickman-type limit.
\newblock \emph{J. Appl. Probab.}, 46\penalty0 (3):\penalty0 732--755, 2009.


\bibitem[Dia(2010)]{dia_thesis}
E.~H.~A. Dia.
\newblock \emph{Exotic options under exponential {L}{\'e}vy model}.
\newblock PhD thesis, Doctoral thesis, Universit{\'e} Paris-Est, 2010.

\bibitem[Dia and Lamberton(2011)]{dia_lamberton2011}
E.~H.~A. Dia, D.~Lamberton, et~al.
\newblock Connecting discrete and continuous lookback or hindsight options in
  exponential {L}{\'e}vy models.
\newblock \emph{Adv. Appl. Probab.}, 43\penalty0 (4):\penalty0 1136--1165,
  2011.

\bibitem[Doney(2007)]{doney_fluctuations}
R.~A. Doney.
\newblock \emph{Fluctuation theory for {L}\'evy processes}, volume 1897 of
  \emph{Lecture Notes in Mathematics}.
\newblock Springer, Berlin, 2007.

\bibitem[Doney and Maller(2002)]{doney_maller_stability}
R.~A. Doney and R.~A. Maller.
\newblock Stability and attraction to normality for {L}\'evy processes at zero
  and at infinity.
\newblock \emph{J. Theoret. Probab.}, 15\penalty0 (3):\penalty0 751--792, 2002.

\bibitem[Duquesne(2003)]{duquesne}
T.~Duquesne.
\newblock Path decompositions for real {L}{\'e}vy processes.
\newblock \emph{Ann. Inst. H. Poincar\'e Probab. Statist.}, 39\penalty0
  (2):\penalty0 339--370, 2003.


\bibitem[Embrechts and Maejima(2002)]{embrechts_selfsimilar}
P.~Embrechts and M.~Maejima.
\newblock \emph{Selfsimilar processes}.
\newblock Princeton Series in Applied Mathematics. Princeton University Press,
  Princeton, NJ, 2002.

\bibitem[Ethier and Kurtz(1986)]{ethier_kurtz}
S.~N. Ethier and T.~G. Kurtz.
\newblock \emph{Markov processes}.
\newblock Wiley Series in Probability and Mathematical Statistics: Probability
  and Mathematical Statistics. John Wiley \& Sons, Inc., New York, 1986.

\bibitem[Feller(1966)]{feller}
W.~Feller.
\newblock \emph{An introduction to probability theory and its applications.
  {V}ol. {II}}.
\newblock John Wiley \& Sons, Inc., New York-London-Sydney, 1966.

\bibitem[Ferreiro-Castilla et~al.(2014)Ferreiro-Castilla, Kyprianou, Scheichl,
  and Suryanarayana]{kyprianou_simulation}
A.~Ferreiro-Castilla, A.~E. Kyprianou, R.~Scheichl, and G.~Suryanarayana.
\newblock Multilevel {M}onte {C}arlo simulation for {L}\'evy processes based on
  the {W}iener-{H}opf factorisation.
\newblock \emph{Stochastic Process. Appl.}, 124\penalty0 (2):\penalty0
  985--1010, 2014.

\bibitem[Gnedenko and Kolmogorov(1954)]{gnedenko_kolmogorov}
B.~V. Gnedenko and A.~N. Kolmogorov.
\newblock \emph{Limit distributions for sums of independent random variables}.
\newblock Addison-Wesley Publishing Company, Inc., Cambridge, Mass., 1954.

\bibitem[Greenwood and Pitman(1980)]{greenwood_pitman}
P.~Greenwood and J.~Pitman.
\newblock Fluctuation identities for {L}\'evy processes and splitting at the
  maximum.
\newblock \emph{Adv. in Appl. Probab.}, 12\penalty0 (4):\penalty0 893--902,
  1980.

\bibitem[Hirano(2001)]{hirano_conditioning}
K.~Hirano.
\newblock L\'evy processes with negative drift conditioned to stay positive.
\newblock \emph{Tokyo J. Math.}, 24\penalty0 (1):\penalty0 291--308, 2001.

\bibitem[Jacod and Protter(2012)]{jacod_protter}
J.~Jacod and P.~Protter.
\newblock \emph{Discretization of processes}, volume~67 of \emph{Stochastic
  Modelling and Applied Probability}.
\newblock Springer, Heidelberg, 2012.


\bibitem[Jacod and Shiryaev(1987)]{jacod_shiryaev}
J.~Jacod and A.~N. Shiryaev.
\newblock \emph{Limit theorems for stochastic processes}, volume 288 of
  \emph{Grundlehren der Mathematischen Wissenschaften [Fundamental Principles
  of Mathematical Sciences]}.
\newblock Springer-Verlag, Berlin, 1987.


\bibitem[Janssen and Van~Leeuwaarden(2009)]{hanssen_leeuwaarden}
A.~J. E.~M. Janssen and J.~S.~H. Van~Leeuwaarden.
\newblock Equidistant sampling for the maximum of a {B}rownian motion with
  drift on a finite horizon.
\newblock \emph{Electron. Commun. Probab.}, 14:\penalty0 143--150, 2009.


\bibitem[Kallenberg(2002)]{kallenberg}
O.~Kallenberg.
\newblock \emph{Foundations of modern probability}.
\newblock Probability and its Applications (New York). Springer-Verlag, New
  York, second edition, 2002.


\bibitem[Kosulajeff(1937)]{kosulajeff}
P.~Kosulajeff.
\newblock Sur la r\'epartition de la partie fractionnaire d'une variable.
\newblock \emph{Math. Sbornik}, 2\penalty0 (5):\penalty0 1017--1019, 1937.

\bibitem[Kuznetsov(2010)]{kuznetsov_WH}
A.~Kuznetsov.
\newblock Wiener-{H}opf factorization and distribution of extrema for a family
  of {L}\'evy processes.
\newblock \emph{Ann. Appl. Probab.}, 20\penalty0 (5):\penalty0 1801--1830,
  2010.


\bibitem[Kwa\'snicki et~al.(2013{\natexlab{a}})Kwa\'snicki, Ma{\l}ecki, and
  Ryznar]{subordinatedBM}
M.~Kwa\'snicki, J.~Ma{\l}ecki, and M.~Ryznar.
\newblock First passage times for subordinate {B}rownian motions.
\newblock \emph{Stochastic Process. Appl.}, 123\penalty0 (5):\penalty0
  1820--1850, 2013{\natexlab{a}}.


\bibitem[Kwa\'snicki et~al.(2013{\natexlab{b}})Kwa\'snicki, Ma{\l}ecki, and
  Ryznar]{suprema_bounds}
M.~Kwa\'snicki, J.~Ma{\l}ecki, and M.~Ryznar.
\newblock Suprema of {L}\'evy processes.
\newblock \emph{Ann. Probab.}, 41\penalty0 (3B):\penalty0 2047--2065,
  2013{\natexlab{b}}.


\bibitem[Kyprianou(2006)]{kyprianou}
A.~E. Kyprianou.
\newblock \emph{Introductory lectures on fluctuations of {L}\'evy processes
  with applications}.
\newblock Universitext. Springer-Verlag, Berlin, 2006.


\bibitem[Lamperti(1962)]{lamperti}
J.~Lamperti.
\newblock Semi-stable stochastic processes.
\newblock \emph{Trans. Amer. Math. Soc.}, 104:\penalty0 62--78, 1962.

\bibitem[Lewis and Mordecki(2008)]{lewis_mordecki}
A.~L. Lewis and E.~Mordecki.
\newblock Wiener-{H}opf factorization for {L}\'evy processes having positive
  jumps with rational transforms.
\newblock \emph{J. Appl. Probab.}, 45\penalty0 (1):\penalty0 118--134, 2008.


\bibitem[Maller and Mason(2008)]{maller_mason}
R.~Maller and D.~M. Mason.
\newblock Convergence in distribution of {L}\'evy processes at small times with
  self-normalization.
\newblock \emph{Acta Sci. Math. (Szeged)}, 74\penalty0 (1-2):\penalty0
  315--347, 2008.


\bibitem[Maller(2009)]{maller_strassen}
R.~A. Maller.
\newblock Small-time versions of {S}trassen's law for {L}\'evy processes.
\newblock \emph{Proc. Lond. Math. Soc. (3)}, 98\penalty0 (2):\penalty0
  531--558, 2009.

\bibitem[Maller(2015)]{maller_trimmed}
R.~A. Maller.
\newblock Strong laws at zero for trimmed {L}\'evy processes.
\newblock \emph{Electron. J. Probab.}, 20:\penalty0 no. 88, 24, 2015.

\bibitem[Michna et~al.(2015)Michna, Palmowski, and
  Pistorius]{palmowski_supremum}
Z.~Michna, Z.~Palmowski, and M.~Pistorius.
\newblock The distribution of the supremum for spectrally asymmetric {L}\'evy
  processes.
\newblock \emph{Electron. Commun. Probab.}, 20:\penalty0 no. 24, 10, 2015.

\bibitem[Piterbarg(1996)]{piterbarg}
V.~I. Piterbarg.
\newblock \emph{Asymptotic methods in the theory of {G}aussian processes and
  fields}, volume 148 of \emph{Translations of Mathematical Monographs}.
\newblock American Mathematical Society, Providence, RI, 1996.


\bibitem[R\'enyi(1958)]{renyi}
A.~R\'enyi.
\newblock On mixing sequences of sets.
\newblock \emph{Acta Math. Acad. Sci. Hungar.}, 9:\penalty0 215--228, 1958.


\bibitem[Resnick(2007)]{resnick_heavy}
S.~I. Resnick.
\newblock \emph{Heavy-tail phenomena}.
\newblock Springer Series in Operations Research and Financial Engineering.
  Springer, New York, 2007.


\bibitem[Samorodnitsky and Taqqu(1994)]{samorodnitsky_taqqu}
G.~Samorodnitsky and M.~S. Taqqu.
\newblock \emph{Stable non-{G}aussian random processes}.
\newblock Stochastic Modeling. Chapman \& Hall, New York, 1994.


\bibitem[Sato(2013)]{sato}
K.-i. Sato.
\newblock \emph{L\'evy processes and infinitely divisible distributions},
  volume~68 of \emph{Cambridge Studies in Advanced Mathematics}.
\newblock Cambridge University Press, Cambridge, 2013.


\bibitem[Shimura(1990)]{shimura}
T.~Shimura.
\newblock The strict domain of attraction of strictly stable law with index
  {$1$}.
\newblock \emph{Japan. J. Math. (N.S.)}, 16\penalty0 (2):\penalty0 351--363,
  1990.

\bibitem[Skorohod(1957)]{skorokhod}
A.~V. Skorohod.
\newblock Limit theorems for stochastic processes with independent increments.
\newblock \emph{Teor. Veroyatnost. i Primenen.}, 2:\penalty0 145--177, 1957.


\bibitem[van~der Vaart(1998)]{vanderVaart}
A.~W. van~der Vaart.
\newblock \emph{Asymptotic statistics}, volume~3 of \emph{Cambridge Series in
  Statistical and Probabilistic Mathematics}.
\newblock Cambridge University Press, Cambridge, 1998.


\bibitem[Whitt(1980)]{whitt}
W.~Whitt.
\newblock Some useful functions for functional limit theorems.
\newblock \emph{Math. Oper. Res.}, 5\penalty0 (1):\penalty0 67--85, 1980.


\bibitem[Whitt(2002)]{whitt_book}
W.~Whitt.
\newblock \emph{Stochastic-process limits}.
\newblock Springer Series in Operations Research. Springer-Verlag, New York,
  2002.

\end{thebibliography}

\appendix
\section{Proofs of the results from \S\ref{SEC:DOMAINS}}
\label{app:proofs}
This Appendix is devoted to proofs of the results from~\S\ref{SEC:DOMAINS}. These proofs make repeated use of Karamata's theorem and its Stieltjes-integral form variant, see~\cite[\S 1.5 and \S 1.6]{bingham_regular}. The corresponding result translated into the setting of regular variation at~0 is stated below, where it is assumed that the intervals of integration include left endpoints and exclude right endpoints.
\begin{theorem}[Karamata's Theorem]
Let $f:(0,\delta)\mapsto\mathbb R_+$ be a positive left-continuous function of bounded variation on compacts.
\begin{itemize}
\item If $f\in\RV_{-\rho}$ and $\varsigma+\rho>0$ then
\begin{equation}\label{eq:k1}\int_x^\delta y^{-\varsigma}\D f(y)/\{x^{-\varsigma} f(x)\}\to -\rho/(\varsigma+\rho).\end{equation}
If~\eqref{eq:k1} holds with $\varsigma+\rho>0$ and $\varsigma>0$ then $f\in\RV_{-\rho}$.
\item If $f\in\RV_{-\rho}$ and $\varsigma+\rho<0$ then
\begin{equation}\label{eq:k2}\int_{0-}^x y^{-\varsigma}\D f(y)/\{x^{-\varsigma} f(x)\}\to \rho/(\varsigma+\rho).\end{equation}
If~\eqref{eq:k2} holds with $\varsigma+\rho<0$ and $\varsigma\neq 0$ then $f\in\RV_{-\rho}$. 
\item If $f\in\RV_{\rho}$ with $\rho>0$ then
\begin{equation}\label{eq:k3}\int_{0}^x y^{-1}f(y)\D y/f(x)\to 1/\rho.\end{equation}
\end{itemize}
\end{theorem}

\begin{proof}[Proof of Theorem~\ref{thm:domains}]
The \levy triplet corresponding to the \levy exponent~$\psi^\epp$ of the rescaled process can be easily identified:
\begin{align*}
\gamma^\epp=\frac{\ep}{a_\ep}\left(\gamma-\int_{a_\ep\leq |x|<1}x\Pi(\D x)\right),\quad {\sigma^\epp}^2=\frac{\ep}{a_\ep^2}\sigma^2,\quad \Pi^\epp(\D x)=\ep\Pi(a_\ep\D x)
\end{align*}
for any $x>0$, assuming that $\ep$ is small enough so that~$a_\ep<1$. 
According to~\cite[Thm.\ 15.14]{kallenberg} the convergence in~\eqref{eq:psi_convergence} is equivalent to
\begin{align}
\label{eq:d} \gamma^\epp-\int_{u<|x|\leq 1}x\Pi^\epp(\D x)=\ep m(u a_\ep)/a_\ep &\to \widehat\gamma-\int_{u<|x|\leq 1}x\widehat \Pi(\D x),\\
\label{eq:sigma}{\sigma^\epp}^2+ \int_{|x|\leq u}x^2\Pi^\epp(\D x)=\ep v(u a_\ep)/a_\ep^2 &\to \widehat\sigma^2+\int_{|x|\leq u}x^2\widehat \Pi(\D x),\\
\label{eq:nu}\Pi^\epp &\vaguely\widehat \Pi
\end{align}
for some (and then for all) $u>0$, where $\int_{u<|x|\leq 1}=-\int_{1<|x|\leq u}$ for $u>1$, and the \levy measure converges vaguely on $[-\infty,0)\cup(0,\infty]$. 

\subsection*{Case (i)} In this case $\widehat\sigma>0,\widehat\gamma=0,\widehat \Pi=0$. Note that $v$ is non-negative and non-decreasing. From~\eqref{eq:sigma} we see that $v$ is necessarily positive and such that $\ep v(u a_\ep)/a_\ep^2\to\widehat \sigma$, where $a_\ep\in\RV_{1/2}$ according to Theorem~\ref{thm:lamperti}. Taking $\ep=1/n$ and noting that $a_{1/n}\sim a_{1/(n+1)}$ (by the uniform convergence theorem) we find according to~\cite[Thm.\ 1.10.3]{bingham_regular} that~$v$ is regularly varying. Since $v(ua_\ep)/v(a_\ep)\to 1$ it must be that $v\in\RV_0$; reference to the above theorem is necessary, because $a_\ep$ is not an arbitrary sequence. Moreover, it is sufficient to choose $a_\ep$ such that $v(a_\ep)/a_\ep^2\sim\widehat \sigma^2\ep^{-1}$, which is always possible according to~\cite[Thm.\ 1.5.12]{bingham_regular}. Furthermore, since 
$\int_x^\infty y^{-2}\D v(y)=\oPi(x)$, we find from~\eqref{eq:k1} that $v\in\RV_0$ is equivalent to $x^2\oPi(x)/v(x)\to 0$ as $x\downarrow 0$.

For sufficiency we need to show that $v\in \RV_0$ implies~\eqref{eq:nu} and~\eqref{eq:d} with $u=1$. 
Observe that
\[\Pi^\epp(\mathbb R\backslash [-x,x])=\ep \oPi(x a_\ep)=\frac{(xa_\ep)^2 \oPi(x a_\ep)}{v(x a_\ep)}\frac{v(x a_\ep)}{x^2v(a_\ep)}\frac{\ep v(a_\ep)}{a_\ep^2}\to 0,\] because the first term goes to 0 while the latter two have finite limits, which shows~$\Pi^\epp\vaguely 0$. Next, we show that $\ep m(a_\ep)/a_\ep\to 0$. Note that $\ep/a_\ep\to 0$ and so it is enough to establish that
\[\frac{\ep}{a_\ep}\int_{a_{\ep}\leq|x|<1} |x|\Pi(\D x)=\frac{\ep v(a_\ep)}{a_{\ep}^2}\frac{\int_{a_{\ep}<x<1} x^{-1}\D v(x)}{a_\ep^{-1}v(a_\ep)}\to 0, \]
but the first term has a finite limit and the second converges to~0 according to~\eqref{eq:k1}. 

\subsection*{Case (ii)} In this case $\widehat\gamma>0,\widehat \sigma=0,\widehat \Pi=0$. We have $\ep m(ua_\ep)/a_\ep\to \widehat\gamma$, but the function $m$ is not monotone in general. Nevertheless, for $v\in(1,\infty)$ and small enough $\ep$ we must have 
\begin{align*}&\sup_{u\in[1,v]}\frac{\ep}{a_\ep}|m(u a_\ep)-m(a_\ep)|\leq \sup_{u\in[1,v]}\frac{\ep}{a_\ep}\int_{a_\ep\leq |x|<ua_\ep}|x|\Pi(\D x)\\
&\leq\sup_{u\in[1,v]}u\ep\int_{1\leq |x|<u}\Pi(a_\ep\D x)\leq v \Pi^\epp((-v,1)\cup(1,v))\to 0.\end{align*}
This and a similar statement for $v\in(0,1)$ lead to the conclusion that 
\[\frac{\ep}{a_\ep}\frac{m(ua_\ep)}{\widehat\gamma}\to 1\qquad\text{ uniformly in }u\text{ on compact sets of }(0,\infty).\]
Since $a_\ep\in\RV_{1}$ we have $a_{1/n}\sim a_{1/(n+1)}$ showing, in particular, that $m(x)/\widehat\gamma$ is positive for all small~$x$. Thus $m(x)/\widehat\gamma\in\RV_0$ according to~\cite[1.9.3]{bingham_regular}, and we may choose $a_\ep$ as stated. 
Moreover,
\begin{equation}\label{eq:d_nu}\ep \oPi(xa_\ep)=\frac{xa_\ep \oPi(xa_\ep)}{m(xa_\ep)}\frac{m(xa_\ep)}{xm(a_\ep)}\frac{m(a_\ep)\ep}{a_\ep}\to 0
\end{equation}
showing that $xa_\ep \oPi(xa_\ep)/m(xa_\ep)\to 0$ uniformly in~$x$ on compact sets of~$(0,\infty)$. So we may conclude that $x\oPi(x)/m(x)\to 0$ as $x\downarrow 0$.
Let us now show that $x\oPi(x)/m(x)\to 0$ and the fact that $m(x)/\widehat\gamma$ is eventually positive imply that $m(x)/\widehat\gamma\in\RV_0$. Observe that $\D m(y)=y(\Pi(\D y)-\Pi(-\D y))$ and so
\[\frac{1}{x^{-1}m(x)}\int_x^{b} y^{-1}\D m(y)=\frac{x}{m(x)}(\Pi(x,b)-\Pi(-b,-x))\leq \frac{x\oPi(x)}{|m(x)|}\to 0\] establishing the claim in view of~\eqref{eq:k1}.

For sufficiency it is only left to show that $\ep v(a_\ep)/a_\ep^2\to 0$, where necessarily $\sigma=0$ in view of $\ep/a_\ep^2\to \infty$ or~\eqref{eq:types}. Hence we need to establish that
\begin{equation}\label{eq:d_toprove}\frac{\ep}{a_\ep^2}\int_0^{a_\ep}x^2\D \oPi(x)=\ep \oPi(a_\ep)-2\frac{\ep}{a_\ep^2}\int_0^{a_\ep}x \oPi(x)\D x\to 0,\end{equation}
where we relied on the fact that $x^2\oPi(x)\to 0$ which follows from $x\oPi(x)/m(x)\to 0$ and $xm(x)\to 0$ as $x\downarrow 0$.
But
\[\frac{\ep}{a_\ep^2}\int_0^{a_\ep}m(x)\D x=\frac{\ep m(a_\ep)}{a_\ep}\int_0^{a_\ep}m(x)\D x/(a_\ep m(a_\ep))\to \widehat\gamma,\]
because the second term converges to~1 according to~\eqref{eq:k3}.
From this and $x\oPi(x)/m(x)\to 0$, as well as~\eqref{eq:d_nu}, we find that~\eqref{eq:d_toprove} indeed holds true.

\subsection*{Case (iii)} In this case $\widehat \sigma=0$ and $\widehat \Pi(\D x)$ is given in~\eqref{eq:psi_stable}. Without loss of generality we assume that $\widehat c_+>0$. The necessity of (a) follows from the Convergence to Types Lemma~\ref{lem:types} and the results in (i) and~(ii).

Concerning (b) we find from~\eqref{eq:nu} that 
\begin{equation}\label{eq:rv_stable}\Pi^\epp(x,\infty)=\ep \oPi_+(xa_\ep)\to\frac{\widehat c_+}{\a}x^{-\a}=\widehat \Pi(x,\infty)\end{equation}
for all $x>0$,
together with the analogous statement for~$(-\infty,-x)$. 
Clearly $\oPi_+$ is monotone and positive for small arguments, otherwise~\eqref{eq:rv_stable} can not hold. Furthermore, $\oPi_+(xa_\ep)/\oPi_+(a_\ep)\to x^{-\a}$ and thus it must be that $\oPi_+\in RV_{-a}$, see~\cite[Thm.\ 1.10.3]{bingham_regular}. Similarly, $\oPi_-\in\RV_{-\a}$ if $\widehat c_->0$, and also $\oPi_-(xa_\ep)/\oPi_+(xa_\ep)\to \widehat c_-/\widehat c_+$ as $\ep\downarrow 0$. The latter convergence is uniform in $x$ on compact sets of $(0,\infty)$, which is inherited from the uniform convergence of $\oPi_\pm(xa_\ep)/\oPi_\pm(a_\ep)$. Since $a_{1/n}\sim a_{1/(n+1)}$, we must have that $\oPi_-(x)/\oPi_+(x)\to \widehat c_-/\widehat c_+$ as $x\downarrow 0$. Furthermore, 
we may always choose $a_\ep$ as stated, and in that case~\eqref{eq:nu} would follow from the conditions in~(b), which will be assumed in the following. 

With respect to~\eqref{eq:sigma} we find that indeed
\[\frac{\ep}{a_\ep^2} v(a_\ep)=-\frac{\ep}{a_\ep^2}\int_0^{a_\ep} x^2\D \oPi(x)\to \frac{\widehat c_++\widehat c_-}{2-\a}=\int_{|x|<1}x^2\widehat\Pi(\D x),\]
because $\ep \oPi(a_\ep)\to (\widehat c_++\widehat c_-)/\a$ and according to~\eqref{eq:k2} also
\[-\int_0^{a_\ep}x^2\D \oPi(x)/(a_\ep^2\oPi(a_\ep))\to\a/(2-\a).\]

For $\a\neq 1$ it is left to show~\eqref{eq:d} for $u=1$, i.e., that
\begin{equation}\label{eq:stable_toprove}\ep m(a_\ep)/a_\ep\to \widehat\gamma=\frac{\widehat c_+-\widehat c_-}{1-\a}.\end{equation}
If $\a\in(0,1)$ then $\gamma'=0$ and so
\[\frac{\ep}{a_\ep}m(a_\ep)=\frac{\ep}{a_\ep}\int_{|x|<a_\ep}x\Pi(\D x)=-\frac{\ep}{a_\ep}\int_0^{a_\ep}x\D (\oPi_+(x)-\oPi_-(x)).\]
If $\widehat c_\pm>0$ then~\eqref{eq:stable_toprove} follows from~\eqref{eq:k2} applied to $\oPi_\pm$ separately. If $\widehat c_-=0$ then we apply that result to $\oPi_+-\oPi_-\in\RV_{-\a}$ and note that $\ep(\oPi_+(a_\ep)-\oPi_-(a_\ep))\to \widehat c_+/\a$.
If $\a\in(1,2)$ then $\ep/a_\ep\to 0$ since $a_\ep\in \RV_{1/\a}$. Moreover, 
\[-\frac{\ep}{a_\ep}\int_{a_\ep\leq |x|<1}x\Pi(\D x)=\frac{\ep}{a_\ep}\int_{a_\ep}^1 x\D (\oPi_+(x)-\oPi_-(x))\to \frac{\widehat c_+-\widehat c_-}{1-\a},\] which follows similarly to the case $\a<1$, but using~\eqref{eq:k1}. Hence~\eqref{eq:stable_toprove} is established for $\a\neq 1$.

In the case of $\a=1$ the convergence in~\eqref{eq:d} does not always hold. But since $\ep \oPi_+(a_\ep)\sim \widehat c_+$ we must have
\begin{equation}\label{eq:last_toprove}\frac{\widehat c_+}{a_\ep \oPi_+(a_\ep)}\left(\gamma-\int_{a_\ep\leq |y|<1}y\Pi(\D y)\right)\to \widehat\gamma,\end{equation}
which shows~\eqref{eq:a_is_1} for a particular sequence $a_\ep$. It is left to show that this limit extends to an arbitrary sequence $x\downarrow 0$. Choose $n=n(x)$ to be the largest integer such that $x<a_{1/n}$. Thus $x\geq a_{1/(n+1)}$ and $n\to\infty$ as $x\downarrow 0$. Using monotonicity of various terms we find that the expression in~\eqref{eq:last_toprove} is bounded from above by
\[\frac{\widehat c_+}{a_{1/(n+1)} \oPi_+(a_{1/n})}\left(\gamma-\int_{a_{1/n}}^1 y\Pi(\D y)+\int_{a_{1/(n+1)}}^1 y\Pi(-\D y)\right)\to \widehat\gamma\]
for all large $n$, because $\oPi_+(a_{1/n})\sim \oPi_+(a_{1/(n+1)})$ and
\[\frac{1}{a_{1/(n+1)} \oPi_+(a_{1/(n+1)})}\int_{a_{1/(n+1)}}^{a_{1/n}}y\Pi(\D y)\downarrow 0.\]
A similar lower bound completes the proof.
\end{proof}

\begin{proof}[Proof of Corollary~\ref{cor:BG}]
The case (iii) of Theorem~\ref{thm:domains} is analyzed using standard arguments. For $\oPi\in\RV_{-\a}$ and any small $\delta>0$ we need to show that 
\[-\int_0^1x^{\a+\delta}\D \oPi(x)<\infty,\qquad -\int_0^1x^{\a-\delta}\D \oPi(x)=\infty.\] Convergence of the first integral follows from integration by parts and Potter's bounds. Divergence of the second integral follows from
\[-\int_{y}^1 x^{\a-\delta}\D \oPi(x)/(y^{\a-\delta}\oPi(y))\to \a/\delta\] which is a consequence of~\eqref{eq:k1}.

In case (i) of Theorem~\ref{thm:domains} we need to show that
\[\int_{|x|<1}|x|^{2-\delta}\Pi(\D x)=\int_0^1 x^{-\delta}\D v(x)=\infty\] for any $\delta>0$. Suppose the opposite. Then $V(y)=\int_y^1 x^{-\delta}\D v(x)$ must have a positive limit, and so $V\in\RV_0$. Now
\[-\int_0^x y^\delta\D V(y)=v(x)-v(0)=v(x),\]
because we assumed that $\sigma^2=0$. From~\eqref{eq:k2} it follows that $v(x)/x^\delta V(x)\to 0$ which can not be true since $v/V\in\RV_0$.

In case (ii) assume first that $X$ is b.v., and so $\beta_{BG}\leq 1$.
Define $M(x)=\int_0^x |y|\Pi(\D y)$ and note that $x\oPi(x)/M(x)\to 0$. In view of $\int_x^{\infty}y^{-1}\D M(y)=\oPi(x)$ and~\eqref{eq:k1} we find that $M(x)\in\RV_0$. Similarly to the case (i) we now see that $\int_0^1x^{-\delta}\D M(x)=\infty$ showing that $\beta_{BG}\geq 1-\delta$.
If $X$ is ub.v.\ then $\beta_{BG}\geq 1$ and we let $M(x)=\int_{x}^1 |y|\Pi(\D y)$, which again must be $\RV_0$. But then clearly $-\int_0^1 x^{\delta}\D M(x)<\infty$ showing that $\beta_{BG}\leq 1+\delta$.
\end{proof}

\section{An extension of the law of large numbers}\label{app:classical}
Reconsider~\eqref{eq:sums} for a constant non-zero limit:
\begin{equation}\label{eq:classical_d}
\sum_{i=1}^n\zeta_i/a_n\to\widehat\gamma\neq 0,\qquad n,a_n\to\infty,
\end{equation}
where $\zeta_i$ are i.i.d.\ and convergence is in probability. In order to have a complete picture with respect to zooming out on random walks, see~\S\ref{sec:zooming_out}, we need to find necessary and sufficient conditions for the convergence in~\eqref{eq:classical_d}.
The interesting part, of course, concerns the cases $\e |\zeta_1|=\infty$ and $\e \zeta_1=0$, because otherwise we may simply take $a_n$ proportional to $n$ and apply the law of large numbers. 
For positive $\zeta_1$ this problem is solved by~\cite[Thm.\ VII.7.3]{feller}, whereas the general case is not readily available in the standard textbooks.
Similarly to Theorem~\ref{thm:domains} (ii) one can establish the following result, which complements~\cite[Thm.\ 3.1]{shimura} characterizing the strict domain of attraction of a strictly 1-stable distribution, see also (3.4) therein.
\begin{proposition}
Let $m(x)=\e(\zeta_1;|\zeta_1|\leq x)$. Then~\eqref{eq:classical_d} holds true if and only if $m(x)/\widehat\gamma$ is eventually positive  and $x\p(|\zeta_1|>x)/m(x)\to 0$ as $x\to\infty$, in which case $a_n/m(a_n)\sim n/\widehat\gamma$. 
\end{proposition}
\begin{proof}[Proof Sketch]
According to~\cite[Thm.\ 15.28]{kallenberg} the convergence in~\eqref{eq:classical_d} holds if and only if
\[n\p(\zeta_1 \in a_n\D x)\vaguely 0,\quad n\VAR(\zeta_1;|\zeta_1|\leq u a_n)/a_n^2\to 0,\quad n\e(\zeta_1;|\zeta_1|\leq u a_n )/a_n\to \widehat\gamma\]
for some (and then for all) $u>0$.
The rest of the proof is somewhat similar to Case (ii) in Appendix~\ref{app:proofs}.
\end{proof}

Assume that $\e \zeta_1=\pm\infty$ then $m(x)\to \pm\infty$ and thus $a_n/n\to \infty$, i.e., the scaling should be faster than linear. Hence if~\eqref{eq:classical_d} holds then $\zeta_i$ can be replaced by $\zeta_i-d$ for any $d\in\mathbb \R$ without changing the limit result. In other words, shifting is irrelevant in this case. 
An example is given by the Pareto distribution with shape~$1$: $\p(\zeta_1\in\D x)=x^{-2}\D x$ for $x>1$, where $m(x)=\log x$. 

\section{Proof of the invariance principle}\label{app:convergence}

\begin{proof}[Proof of Theorem~\ref{thm:convergence}]
The proof consists of three steps, where in steps (ii) and (iii) we use some particular representations of the laws $\p^\downarrow_x$ and $\p^\downarrow$ avoiding double limits. In the following we define some quantities for the process~$X$ and assume that the analogous quantities are defined for each $\Xn$ without explicitly writing them.

(i) Consider the (weak) ascending ladder processes $(L^{-1},H)$, where $L^{-1}$ denotes the inverse local time at the supremum and $H_t=X_{L_t^{-1}}$. The corresponding Laplace exponent is denoted by $k(\a,\beta)$ and normalized so that $k(1,0)=1$, see~\cite[Ch.\ VI]{bertoin} or~\cite[Ch.\ 6]{kyprianou}. 
By the continuous mapping theorem we get convergence of the Wiener-Hopf factors in~\cite[Thm.\ 6.16(ii)]{kyprianou}, which then implies convergence of the bivariate exponents $k^{(n)}(\a,\b)\rightarrow k(\a,\b)$ (and hence also weak convergence of the ladder processes).
It is noted that in the above textbooks the results are formulated for non-killed \levy processes, but they extend to killed \levy processes in a straightforward way. 

(ii) The following representation of the semigroup of the conditioned process is standard~\citep{chaumont}:
\begin{equation}\label{eq:doob_h}
p^\downarrow_{t}(x,\D y)=\frac{m(-y)}{m(-x)}\p_{x}(X_t\in\D y,\overline X_t<0),\qquad x<0,
\end{equation}
where $\overline X_t=\sup_{s\leq t}X_s$ and $m(r)=\e \int_0^\infty\ind{H_t<r}\D t$
is a finite, increasing function on $(0,\infty)$. Since we assumed that $X$ is not a compound Poisson process, the function $m$ is continuous and $\p_x(\overline X_t=0)=0$ for $x<0$. Hence we have
\[\pn_{x}(X_t\in\D y,\overline X_t<0)\Rightarrow\p_{x}(X_t\in\D y,\overline X_t<0).\]
It is well-known and easy to see that $\int_{[0,\infty)} e^{-\b x} \D m(x)=1/k(0,\b)$ for $\b>0$, because $\D m(x)$ is the potential measure of the ladder height process. Thus according to step (i) the Laplace transform of $\D h^{(n)}(x)$ converges to that of $\D m(x)$ for all $\beta>0$, and so the corresponding cumulative distribution functions converge: $m^{(n)}(x)\rightarrow m(x)$, because the latter is continuous~\cite[Thm.\ XIII.1.2a]{feller}.
We have established convergence of the semigroup given in~\eqref{eq:doob_h}, and so according to~\cite[Thm.\ 4.2.5]{ethier_kurtz} we obtain
\[\pn_x^\downarrow\Rightarrow\p_x^\downarrow,\qquad \text{for }x<0\] 
because the corresponding processes are Feller and the initial distributions coincide, see also~\cite[Lem.\ 4.2.3]{ethier_kurtz} concerning the one-point compactification of~$\R$.

(iii) Finally, we recall~\citep[Thm.\ 1]{chaumont} that $\p^\downarrow$ is also the law of the post-supremum process under $\p_x^\downarrow$ for any $x<0$. Under the latter law the time of the supremum is finite and unique, and so we can apply the continuous mapping theorem to establish that
\[\pn^\downarrow\Rightarrow\p^\downarrow.\]
The respective map is continuous at any $\omega$ such that the time of supremum of $(\omega_t)_{t\geq 0}$ is finite and unique. Indeed, for a sequence $\omega^{(n)}$ converging to $\omega$ the corresponding suprema and their (last) times will converge. Then it is easy to see that the post-supremum processes converge in Skorokhod topology given that the initial evolution of paths can be matched. The latter follows from the assumption that $\omega_{0-}=\omega_{0-}^{(n)}=0$ allowing to deal with the case when the post-supremum process starts at a negative value. The proof is complete.
\end{proof}

\end{document}